\documentclass[a4paper,10pt]{amsart}
\usepackage[utf8]{inputenc}
\usepackage[english]{babel}

\usepackage{amsfonts}
\usepackage{extarrows}
\usepackage{longtable}
\usepackage{mathtools,leftindex,tensor,mhchem}
\usepackage[english]{babel}
\usepackage{amsmath,amssymb,verbatim,mathrsfs,latexsym,paralist}
\usepackage{graphicx}
\usepackage{epstopdf}
\usepackage{indentfirst}
\usepackage{multirow}
\usepackage{amsthm}
\usepackage{longtable}
\allowdisplaybreaks[4]
\newcommand\restr[2]{{% we make the whole thing an ordinary symbol
  \left.\kern-\nulldelimiterspace % automatically resize the bar with \right
  #1 % the function
  \vphantom{\big|} % pretend it's a little taller at normal size
  \right|_{#2} % this is the delimiter
  }}

\newcommand{\testleftlong}{\longleftarrow\!\shortmid}
\usepackage{tikz}
\usetikzlibrary{arrows,decorations.markings}
\usetikzlibrary{matrix}
\usetikzlibrary{graphs}
\usetikzlibrary{backgrounds}

\newcommand{\C}{\mathbb{C}}

\newcommand{\Q}{\mathbb{Q}}

\newcommand{\GL}{\mathrm{GL}}

\newcommand{\Sp}{\mathrm{Sp}}

\newcommand{\mf}{\mathfrak}
\newcommand{\g}{\mf{g}}
\newcommand{\h}{\mf{h}}

\newcommand{\so}{\mf{so}}
\renewcommand{\sp}{\mf{sp}}
\newcommand{\gl}{\mf{gl}}
\newcommand{\ssl}{\mf{sl}}
\newcommand{\uu}{\mf{u}}
\renewcommand{\a}{\mf{a}}
\renewcommand{\c}{\mf{c}}
\renewcommand{\o}{\mf{o}}
\newcommand{\wtt}[3]{{\small (\scriptstyle{#1};\scriptstyle{\mathbf{#2}};\scriptstyle{#3})}}
\newcommand{\wts}[2]{{({\scriptstyle{#1}};{\scriptstyle{#2}})}}

\newcommand{\z}{\mf{z}}
\newcommand{\ttt}{\mf{t}}

\newcommand{\End}{\mathrm{End}}
\newcommand{\ad}{\mathrm{ad}}
\newcommand{\G}{\mathcal{G}}
\renewcommand{\ggg}{\mathrm{g}}
\renewcommand{\O}{\mathrm{O}}
\newcommand{\SO}{\mathrm{SO}}
\renewcommand{\Sp}{\mathrm{Sp}}
\newcommand{\Aut}{\mathrm{Aut}}

\newcommand{\wG}{\widehat{G}}

\numberwithin{equation}{section}
\newtheorem{theorem}{Theorem}[section]

\newtheorem{lemma}[theorem]{Lemma}

\theoremstyle{remark}
 
\theoremstyle{remark}
\newtheorem{rmk}[theorem]{Remark}
\newtheorem{example}[theorem]{Example}   

\title[Computing component groups]{Computing component groups of stabilizers of nilpotent orbit representatives}
\author{Emanuele Di Bella}
\address{
Dipartimento di Matematica\\
Universit\`{a} di Trento\\
Italy}
\email{emanuele.dibella@unitn.it}
\author{Willem A. de Graaf}
\address{
Dipartimento di Matematica\\
Universit\`{a} di Trento\\
Italy}
\email{willem.degraaf@unitn.it}
\date{}

\begin{document}

%%%%%%%%%%%%%%%%%%%%%%%Stuff for Dynkin diag - copied from somewhere%%%

% Max' Tikz styles for dynkin diagrams
\def\DynkinNodeSize{1.5mm}
\def\DynkinArrowLength{1.5mm}
\tikzset{
% a diagram node
dnode/.style={
circle,
inner sep=0pt,
minimum size=\DynkinNodeSize,
fill=white,
draw},
middlearrow/.style={
decoration={markings,
mark=at position 0.6 with
%{\arrow[black]{angle 90};}
%{\arrow[black]{angle 60};}
%{\arrow[black]{stealth};}
{\draw (0:0mm) -- +(+135:\DynkinArrowLength); \draw (0:0mm) -- +(-135:\DynkinArrowLength);},
},
postaction={decorate}
},
leftrightarrow/.style={
decoration={markings,
mark=at position 0.999 with
{
\draw (0:0mm) -- +(+135:\DynkinArrowLength); \draw (0:0mm) -- +(-135:\DynkinArrowLength);
},
mark=at position 0.001 with
{
\draw (0:0mm) -- +(+45:\DynkinArrowLength); \draw (0:0mm) -- +(-45:\DynkinArrowLength);
},
},
postaction={decorate}
},
% single edge
sedge/.style={
},
% directed double edge
dedge/.style={
middlearrow,
double distance=0.5mm,
},
% directed triple edge
tedge/.style={
middlearrow,
double distance=1.0mm+\pgflinewidth,
postaction={draw}, % third line
},
% double edge with two arrows, for \tilde{A}_1 residues
infedge/.style={
leftrightarrow,
double distance=0.5mm,
},
}

%%%%%%%%%%%%%%%%%%%%%%%%%%%End stuff for Dyn diag%%%%%%%%%%%%%%%%%%%%%%%%%%

\begin{abstract}
We describe computational methods for computing the component group of the
stabilizer of a nilpotent element in a complex simple Lie algebra. Our
algorithms have been implemented in the language of the computer algebra system
{\sf GAP}4. Occasionally we need Gr\"obner basis computations; for
these we use the systems {\sc Magma} and {\sc Singular}.
The resulting component groups have been
made available in the {\sf GAP}4 package {\sf SLA}. 
\end{abstract}

\maketitle

\section{Introduction}

The theory of the nilpotent orbits of simple Lie algebras has seen tremendous
developments over the past decades. For overviews we refer to the book by
Collingwood and McGovern (\cite{colmcgov}) and the survey paper by Jantzen
(\cite{jannilp}). In many contexts an important role is played by the
stabilizer of a nilpotent element. Examples are the Springer correspondence
(\cite[\S 10.1]{colmcgov}, \cite[\S 13]{jannilp}) and the theory of
$W$-algebras (see, for example, \cite{losevW}). 
In characteristic 0 the identity component
of this stabilizer is determined by its Lie algebra, which is the centralizer
in the Lie algebra and thus can be determined by linear algebra. So what
remains is to determine the component group.

For the Lie algebras of classical
type these component groups can be determined by general considerations,
see \cite[\S 3]{jannilp} and Section \ref{sec:classical} below.
For the Lie algebras of exceptonal type the component groups have first been
determined by Alekseevskii, \cite{alekseevski} and later by Sommers,
\cite{sommers}. These references provide methods to determine the
isomorphism type of the component group, not directly the component group
itself. More precisely, with
the construction in \cite{sommers} it is possible to determine
an element in each conjugacy class of the component group; but these elements
lie in the stabilizer of different nilpotent elements in the same orbit.
In order to perform explicit computations with the component group one
often needs explicit generators of the component group of the stabilizer of
one fixed nilpotent element. In \cite{lawtest} such generators have been
determined by impressive hand calculations for the nilpotent orbits in the
Lie algebras of exceptional type. However, these use the specific form
of the chosen representative and their description depends on a particular
choice of the multiplication table of the simple Lie algebra at hand. Thus
they are not immediately usable for a different representative
or a Lie algebra given by a different multiplication table.

In this paper we develop computational methods to obtain the component group
of the centralizer of a nilpotent element in a simple Lie algebra
over $\C$. For the classical types there is a straightforward method that
directly translates the theoretical construction, see Section
\ref{sec:classical}. In Section \ref{sec:double} we devise a method for the
exceptional types using the double centralizer of an $\ssl_2$-triple. 
Our method gives an independent construction of the component
group, that is, it does not depend on the prior knowledge of the size
or isomorphism type of the component group. It is based on the following
observation: it is a basic fact that every nilpotent element is contained
in an $\ssl_2$-triple and that the component group of the centralizer of the
nilpotent element is equal to the component group of the centralizer of the
$\ssl_2$-triple (see Section \ref{sec:nilporb}). Let $\c$ denote the sum of
the centralizer of such an $\ssl_2$-triple and its double centralizer (that is,
the centralizer of the centralizer). Let $V$ denote a complement to $\c$
that is perpendicular to $\c$ with respect to the Killing form. Then by
direct case by case computation we established that $V$ is a direct sum
of pairwise non-isomorphic $\c$-modules. We then exploit the symmetries that
exist among the highest weights of the irreducible summands of $V$ along with
the symmetries of $\c$ (that is, its automorphisms) to construct all elements
of the component group of the centralizer of the $\ssl_2$-triple. 

A technical ingredient, described in Section \ref{sec:conj},
is an algorithm for finding an explicit automorphism conjugating two
nilpotent elements that lie in the same orbit. We start with a section with
preliminary material on semisimple Lie algebras and nilpotent orbits.

We have implemented our algorithms in the computational algebra system
{\sf GAP}4, \cite{gap4}, with the help of the {\sf SLA} package, \cite{sla}.
For the algorithm of Sections \ref{sec:conj}, \ref{sec:double} we need
Gr\"obner basis computations. For these we have used the systems {\sc Magma},
\cite{magma} and {\sc Singular}, \cite{DGPS}. For an introduction into
the technique of Gr\"obner bases we refer to \cite{clo}.
For the exceptional types we have constructed the component groups for
all nilpotent orbits, using the representatives given in \cite{gra14}.
The component groups have been made available in the latest release (version
1.6.2) of the {\sf SLA} package, \cite{sla}.

Throughout the paper we use basic notions of the theory of Lie algebras and
algebraic groups. For introductions into these we refer to the books
\cite{borel}, \cite{hum2}, \cite{hum}.

\section{Preliminaries}\label{sec:prelim}

\subsection{Automorphisms of a semisimple Lie algebra}\label{sec:pre1}

Let $\g$ be a semisimple Lie algebra over $\C$ and fix a Cartan subalgebra
$\h$ of $\g$. Let $\Phi\subset \h^*$ be the root system of $\g$ with respect
to $\h$. By $\ell = \dim\h$ we denote the rank of $\g$.
For a root $\alpha\in \Phi$ we denote the corresponding root
space by $\g_\alpha$, that is
$$\g_\alpha = \{ x\in \g\mid [h,x]=\alpha(h)x \text{ for all } h\in \h\}.$$

Let $\Delta=\{\alpha_1,\ldots,\alpha_\ell\}\subset \Phi$ be a basis of simple
roots. For $\alpha,\beta\in \Phi$ we write $\langle \alpha,\beta^\vee\rangle=
\tfrac{2(\alpha,\beta)}{(\beta,\beta)}$. The matrix $(\langle\alpha_i,
\alpha_j^\vee\rangle)_{i,j=1}^\ell$ is called the Cartan matrix of $\Phi$.

There are nonzero $h_i\in \h$, $x_{\alpha_i} \in \g_{\alpha_i}$, $x_{-\alpha_i}\in
\g_{-\alpha_i}$ for $1\leq i\leq\ell$ such that
\begin{equation}\label{eq:cangen}
[h_i,h_j]=0,\, [x_{\alpha_i},x_{-\alpha_j}] = \delta_{ij} h_i,\,
[h_j,x_{\pm\alpha_i}] = \pm\langle \alpha_i,\alpha_j^\vee\rangle x_{\pm \alpha_i}.
\end{equation}
        
These generate $\g$ and any set of elements of $\g$ satisfying
\eqref{eq:cangen} is called a {\em canonical generating set} of $\g$.

Let $\hat h_i^k$, $\hat x_{\pm\beta_i}^k$ for $k=1,2$ be two canonical generating
sets. That means that we only require that both sets satisfy the relations
\eqref{eq:cangen} with $\hat h_i^k$ in place of $h_i$ and $\hat x_{\pm \beta_i}^k$
in place of $x_{\pm\alpha_i}$. Then mapping $h_i^1 \mapsto h_i^2$,
$\hat x_{\pm\beta_i}^1\mapsto
\hat x_{\pm\beta_i}^2$ extends to a unique automorphism of $\g$
(\cite[Chapter IV, Theorem 3]{jac}).

Now let $h_i$, $x_{\pm\alpha_i}$, $1\leq i\leq \ell$ be a fixed canonical
generating set, with $x_{\pm\alpha_i}\in \g_{\pm\alpha_i}$ as above. 
Let $\pi$ be a permutation of
$\{1,\ldots,\ell\}$ such that $\langle \alpha_i,\alpha_j^\vee\rangle =
\langle \alpha_{\pi(i)},\alpha_{\pi(j)}^\vee\rangle$ for all $i,j$. Then
it follows that there is a unique automorphism $\sigma_{\pi}$ of
$\g$ such that $\sigma_\pi(h_i) = h_{\pi(i)}$, $\sigma_\pi(x_{\pm\alpha_i})
= x_{\pm \alpha_{\pi(i)}}$. Such an automorphism is called a {\em diagram
automorphism} of $\g$. It is clear that $\sigma_{\pi_1\pi_2}= \sigma_{\pi_1}
\sigma_{\pi_2}$. Hence we get a finite group of diagram automorphisms of $\g$
denoted $\Gamma$.

An $x\in \g$ is said to be nilpotent if the map $\ad x : \g\to\g$, with
$\ad x(y) =[x,y]$, is nilpotent. If $x\in \g$ is nilpotent then $\exp(\ad x)$
is an automorphism of $\g$. Let $G$ be the group generated by all such
automorphisms. Then $G$ is a connected algebraic subgroup of $\GL(\g)$ whose
Lie algebra is equal to $\ad \g =\{ \ad y\mid y\in \g\}$. The group $G$ is
called the {\em inner automorphism group} of $\g$. Furthermore we have
$$\Aut(\g) = G\rtimes \Gamma$$
(\cite[\S IX.4]{jac}, \cite[\S 4, Theorem 1]{onishchik}). Hence $G$ is the
identity component of the algebraic group $\Aut(G)$ and its components are
$G\gamma$ where $\gamma$ runs over $\Gamma$.

\begin{rmk}\label{rem:outer}
Let $\sigma$ be an automorphism of finite order of $\g$. Let
$\g_0 = \{ x\in \g\mid \sigma(x)=x\}$. Then $\sigma$ is inner if and only
if $\g_0$ has a Cartan subalgebra that is also a Cartan subalgebra of $\g$.
(This is well known; it can be shown by observing that $\sigma$ is semisimple,
hence if $\sigma$ is inner it must lie in a maximal torus of $G$, whose Lie algebra
is a Cartan subalgebra which is pointwise fixed by $\sigma$).
We see that for finite order automorphisms we have a
straightforward algorithm for deciding whether they are inner or not.

There is a more general algorithm to check whether a given element of
$\GL(n,\C)$ lies in a connected algebraic subgroup $H$ of $\GL(n,\C)$,
using only the Lie algebra of $H$. This method is outlined in
\cite[Remark 5.8]{borwdg}. We need this algorithm because some of our
constructions may give too many elements of the component group, and we need
a method
to test whether two elements are equal modulo the identity component.
\end{rmk}

\subsection{Nilpotent orbits}\label{sec:nilporb}

Let $e\in \g$ be nilpotent. Let $G$ be the inner automorphism group of $\g$.
Then the orbit $G\cdot e$ consists of nilpotent
elements and is therefore called a {\em nilpotent orbit} of $\g$. We refer to
\cite{colmcgov}, \cite{jannilp} for overviews of the theory of such orbits.
The Jacobson-Morozov theorem states that for a nilpotent $e\in \g$ there are
$h,f$ such that
$$[h,e]=2e,\, [h,f]=-2f,\, [e,f]=h.$$
The triple $(h,e,f)$ is called an $\ssl_2$-triple.

Let $h_1,\ldots,h_\ell$ be the elements of a canonical generating set, as in
Section \ref{sec:pre1}. Let $\h_\Q$ be the $\Q$-vector space spanned by
$h_1,\ldots, h_\ell$. Then $\h_\Q$ is equal to the space of all $h\in \h$
having rational eigenvalues. The element $\ad h$
has integral eigenvalues, and therefore $h\in \h_\Q$. 

Let $e,e'\in \g$ be nilpotent lying in $\ssl_2$-triples $(h,e,f)$,
$(h',e',f')$. Then $e,e'$ lie in
the same $G$-orbit if and only if there is a $\sigma\in G$ such that
$\sigma(h)=h'$, $\sigma(e)=e'$, $\sigma(f)=f'$ if and only if
there is a $\sigma\in G$ with $\sigma(h) = h'$ (\cite[Theorem 8.1.4]{gra16}).
We also remark that there are efficient algorithms to decide whether
$e,e'$ lie in the same $G$-orbit (see \cite[\S 8.2]{gra16}). 

Let $e\in \g$ be nilpotent lying in an $\ssl_2$-triple $(h,e,f)$.
Let $Z_G(e) = \{\sigma\in G\mid \sigma(e)=e\}$ denote the stabilizer of
$e$ in $G$. Similarly we set $Z_G(h,e,f) = \{\sigma\in G\mid \sigma(h)=h,\,
\sigma(e)=e,\,\sigma(f)=f\}$. Then we have
$$Z_G(e) = Z_G(h,e,f) \ltimes U$$
where $U$ is a unipotent group. Moreover, representatives of the component
group of $Z_G(h,e,f)$ are also representatives of the component group
of $Z_G(e)$ (\cite[Lemma 3.7.3]{colmcgov}).

Set $\z_\g(h,e,f)=\{x\in\g\mid [x,h]=[x,e]=[x,f]=0\}$. Then $\{\ad x \mid
x\in \z_\g(h,e,f)\}$ is the Lie algebra of $Z_G(h,e,f)$ (this is well-known,
it follows for example from \cite[Corollary 4.2.8]{gra16}). Occasionally we
also say that $\z_\g(h,e,f)$ is the Lie algebra of $Z_G(h,e,f)$. 

\subsection{Irreducible modules of reductive Lie algebras}\label{sec:red}

Here we recall some well known facts on irreducible representations of
reductive Lie algebras. First we recall that a complex Lie algebra $\c$ is
{\em reductive} if $\c = \c'\oplus \z(\c)$ where $\c'=[\c,\c]$ is semisimple and
$\z(\c)$ is the centre of $\c$ (cf. \cite[\S 19.1]{hum}).

Let $\c$ be a reductive Lie algebra and $\rho : \c \to \gl(V)$ a finite
dimensional representation such that $\rho(x)$ is semisimple for all $x\in
\z(\c)$. Then $V$ is completely reducible, that is, it is a direct sum of
irreducible $\c$-modules (see \cite[Chapter III, Theorem 10]{jac}).

Let $\rho : \c \to \gl(V)$ be irreducible such that $\rho(x)$ is semismple  
for all $x\in \z(\c)$. The structure theory of irreducible
representations of semisimple Lie algebras (see, e.g., \cite[\S 20]{hum})
generalizes directly to the case of reductive Lie algebras. Here we give a 
short overview. Let $\a$ be
a Cartan subalgebra of $\c$ then $\a = \a'\oplus \z(\c)$, where $\a'$
is a Cartan subalgebra of $\c'$. We consider the root system of $\c'$ with
respect to $\a'$ and let $h_1,\ldots,h_r$, $x_1,\ldots,x_r$,
$y_1,\ldots,y_r$ be a canonical generating set (here $h_i\in \a'$,
$x_i$, $y_i$ are root vectors corresponding to the simple roots,
repectively the negative simple roots). Let $h_{r+1},\ldots,h_\ell$ be a basis
of $\z(\c)$, so that $h_1,\ldots,h_\ell$ is a basis of $\a$. For $\mu\in
\a^*$ we set
$$V_\mu = \{ v\in V \mid \rho(h)v = \mu(h)v \text{ for all } h\in \a\}.$$
If $V_\mu\neq 0$ then we say that $\mu$ is a {\em weight} of $V$ and
$V_\mu$ is the corresponding {\em weight space}. Any element of $V_\mu$ is
called a {\em weight vector} of weight $\mu$. A weight $\lambda$ of $V$ is
said to be a highest weight if $\rho(x_i)V_\lambda =0$ for $1\leq i\leq r$.
Because $\rho$ is irreducible, $V$ has a unique highest weight $\lambda$
and $\dim V_\lambda=1$. A nonzero $v_0\in V_\lambda$ is called a
{\em highest weight vector} of $V$. Fixing such a $v_0$ we have that the
elements
$$\rho(y_{i_1})\cdots \rho(y_{i_k}) v_0\text{ for } k\geq 0, 1\leq i_j\leq r$$
span $V$.

Let now $\rho$ be not necessarily irreducible, such that $\rho(x)$ is
semisimple for all $x\in \z(\c)$. The above description of a spanning set
of an irreducible module gives a straightforward method to find a decomposition
of $V$ as a direct sum of irreducible modules. This works as follows. First
we compute the space $V_0 = \{ v\in V \mid \rho(x_i)v = 0 \text{ for }
1\leq i\leq r\}$. Secondly, we decompose $V_0$ as a direct sum of weight spaces.
Thirdly, for each weight space we fix a basis. Those basis vectors are
exactly the highest weight vectors of the different irreducible modules in
the decomposition.

We say that $\rho$ (or the $\c$-module $V$) is {\em multiplicity free} if
the dimension of each weight space in $V_0$ is 1. This is the same as
saying that $V$ decomposes as a sum of pairwise non-isomorphic irreducible
modules. In that case this decomposition is uniquely determined.

\subsection{Notation}

Throughout this paper we will freely use the notation introduced in this
section. In particular $\g$ is a simple Lie algebra with Cartan subalgebra
$\h$ and root system $\Phi$ with fixed set of simple roots $\Delta = \{
\alpha_1,\ldots,\alpha_\ell\}$. We let $\ad : \g \to \gl(\g)$ be the
adjoint map of $\g$ and $G$ will be the inner automorphism group of $\g$.
If $\uu$ is a subalgebra of $\g$ and $x\in \uu$ then by $\ad_\uu x :
\uu\to\uu$ we will denote the restriction of $\ad x$ to $\uu$.

For an algebraic group $H$ we denote its identity component by $H^\circ$; hence
its component group is $H/H^\circ$. 

\section{Simple Lie algebras of classical type}\label{sec:classical}

In this section we let $\g$ be a simple Lie algebra of classical type,
that is, of type $A_\ell$ ($\ell\geq 1$), $B_\ell$ ($\ell\geq 2$),
$C_\ell$ ($\ell\geq 3$), $D_\ell$ ($\ell\geq 4$). For the Lie algebras of
type $A_\ell$ the stabilizer $Z_G(h,e,f)$ is always connected (see
\cite[\S 3.5]{sommers}) so we deal with the Lie algebras of type $B$, $C$, $D$.
For these algebras we describe an algorithm to obtain the component group of
$Z_G(h,e,f)$ closely following the construction in \cite[\S 3.1]{jannilp}.

We will focus on the Lie algebras of type $B_\ell$ and $D_\ell$; in Remark
\ref{rem:Cl} we indicate what has to be changed for type $C_\ell$. We will
first work with the groups $O(n,\varphi)$ (defined below) instead of
the inner automorphism group $G$. Later we will pass to $G$. 

Let $V$ be an $n$-dimensional vector space over $\C$. Let $\varphi : V\times
V\to \C$ be a nondegenerate bilinear form. Set
$$\G(V,\varphi) = \{ g\in \GL(V) \mid \varphi(gv,gw) = \varphi(v,w)
\text{ for all } v,w\in V\}.$$
This is an algebraic subgroup of $\GL(V)$ with Lie algebra
$$\ggg(V,\varphi) = \{ x\in \End(V) \mid \varphi(xv,w)+\varphi(v,xw)=0
\text{ for all } v,w\in V\}.$$
The group $\G(V,\varphi)$ acts on the Lie algebra $\ggg(V,\varphi)$ by
$g\cdot x = gxg^{-1}$. This yields a map from $\G(V,\varphi)$
to the automorpism group of $\ggg(V,\varphi)$.

If $\varphi$ is symmetric (that is $\varphi(v,w) = \varphi(w,v)$ for all
$v,w\in V$) then $\G(V,\varphi)=\O(V,\varphi)$ is the orthogonal group and
we write $\o(V,\varphi) = \ggg(V,\varphi)$ for its Lie algebra. If $\varphi$
is alternating (that is, $\varphi(v,w) = -\varphi(w,v)$ for all $v,w\in V$)
then we write $\Sp(V,\varphi) = \G(V,\varphi)$ which is the symplectic
group. In that case we write $\sp(V,\varphi)=\ggg(V,\varphi)$.

In the remainder of this section we let $\varphi$ be symmetric.
If $n\geq 5$ is odd then $\o(V,\varphi)$ is  simple
of type $B_\ell$ where $n=2\ell +1$. If $n\geq 8$ is even then it is simple
of type $D_\ell$ where $n=2\ell$. 

The group $\O(V,\varphi)$ has two connected components. The identity
component is $\mathrm{SO}(V,\varphi) = \{ g\in \O(V,\varphi) \mid
\det(g)=1\}$. The second component is $\{  g\in \O(V,\varphi) \mid
\det(g)=-1\}$. We now describe a simple method to find an element in this
second component. Let $v_1,\ldots,v_n$ be a basis of $V$. Let $w_n\in V$
be such that $\varphi(w_n,w_n)\neq 0$. This element can be found as follows:
if there is an $i$ with $\varphi(v_i,v_i)\neq 0$ then we set $w_n=v_i$.
Otherwise there are $i<j$ with $\varphi(v_i,v_j)\neq 0$ and we set
$w_n=v_i+v_j$. Let $w_1,\ldots,w_{n-1}\in V$ be a basis of the subspace
$\{v\in V \mid \varphi(w_n,v)=0\}$. Then $w_1,\ldots,w_n$ is a basis of $V$.
Let $g\in \GL(V)$ be the map that sends $w_i\mapsto w_i$ for $1\leq i
\leq n-1$ and $w_n\mapsto -w_n$. Then $\varphi(gw_i,gw_j)=\varphi(w_i,w_j)$
(for $i,j\leq n-1$ and $i=j=n$ this is clear, for $i\leq n-1$, $j=n$ both are
0). Hence $g\in \O(V,\varphi)$ and $\det(g)=-1$.

Let $(h,e,f)$ be an $\ssl_2$-triple in $\o(V,\varphi)$ and let $\a$ be the
subalgebra
spanned by them. Then $\a$ acts on $V$ and $V$ splits as a direct sum of
irreducible $\a$-modules, $V=V_1\oplus\cdots \oplus V_m$. Write
$d_i=\dim V_i$. From the representation theory of $\ssl_2$
(\cite[\S II.7]{hum}) it follows that
each $V_i$ has a unique (up to nonzero scalar multiples)
element $\hat v_i$ such that $f\cdot \hat v_i = 0$, $h\cdot \hat v_i=
(-d_i+1)\hat v_i$,
and $\{e^k \cdot \hat v_i\mid 0\leq k\leq d_i-1\}$ forms a basis of $V_i$,
and $e^{d_i} \cdot \hat v_i=0$.
For $s\geq 1$ we let $M_s$ be the space spanned by all $\hat v_i$ such that
$d_i=s$. Note that the direct sum decomposition of $V$ is not unique in
general, but the space $M_s$ is uniquely determined.

Write $\wG = \O(V,\varphi)$. Let $g\in Z_{\wG}(h,e,f)$ then $g$ stabilizes each
space $M_s$. Moreover $g\cdot (e^k\cdot \hat v_i) = e^k (g\cdot \hat v_i)$.
It follows that the action of $g$ on each $M_s$ determines $g$. So we get
an injective homomorphism
\begin{equation}\label{fun}
Z_{\wG}(h,e,f) \to \GL(M_1)\times \GL(M_2) \cdots  .
\end{equation}

On the space $M_s$ define a bilinear form $\psi_s$ by $\psi_s(v,w) =
\varphi(v,e^{s-1}w)$. In \cite[\S 3.7]{jannilp} it is shown that $\psi_s$ is
symmetric if $s$ is odd and alternating if $s$ is even. So for $s$ odd we
get the orthogonal group $\O(M_s,\psi_s)$ and for $s$ even we get the
symplectic group $\Sp(M_s,\psi_s)$. Now \cite[\S 3.8, Proposition 2]{jannilp}
states that the above injective homomorphism is an isomorphism onto
\begin{equation}\label{eq:prod}
\prod_{s \text{ odd}} \O(M_s,\psi_s) \times \prod_{s \text{ even}} \Sp(M_s,\psi_s).
\end{equation}  

For $t$ odd let $\hat g_t$ be a fixed element in the non-identity component
of $\O(M_t,\psi_t)$. Above we have outlined how to find such an element.
Let $g_t$ be the preimage of $1\times \cdots \times 1 \times \hat g_t \times 1
\times \cdots 1$ (where in the product decomposition \eqref{eq:prod} we
put a 1 for all $s\neq t$). Then from the considerations above it follows that
the component group of $Z_{\wG}(h,e,f)$ is the elementary abelian 2-group
generated by $g_t$ for all odd $t$.

Setting $\SO(V,\varphi) = \{g\in \O(V,\varphi) \mid \det(g)=1\}$ we
have that $\SO(V,\varphi)$ is connected and the map $\O(V,\varphi)\to \Aut(
\so(V,\varphi))$ restricts to a surjective map $\SO(V,\varphi) \to G$
(where, as usual, $G$ is the identity component of $\Aut(\so(V,\varphi))$).
It follows that the component group of $\{ g\in Z_{\wG}(h,e,f) \mid \det(g)=1\}$
surjects to the component group of $Z_{G}(h,e,f)$.

This leads to an immediate algorithm for computing the component group of
$Z_G(h,e,f)$ of an $\ssl_2$-triple in $\o(V,\varphi)$, which we summarize
in the following steps:

\begin{enumerate}
\item Let $\a$ be the subalgebra spanned by $h,e,f$. Compute the direct sum
  decomposition of the $\a$-module $V=V_1\oplus \cdots \oplus V_m$.
\item In each $V_i$ compute the unique (up to scalar multiples) $\hat v_i$
  such that $f\cdot \hat v_i=0$.
\item For $s\geq 1$ let $M_s$ be the space spanned by the $\hat v_i$ such
  that $\dim V_i=s$.
\item For each odd $s$ compute the reflection $\hat g_s\in \O(M_s,\psi_s)$ with
  determinant -1.
\item Construct $g_s\in \O(V,\varphi)$ in the following way:
  \begin{enumerate}
  \item For each $i$ with $\dim V_i = s$ compute a basis of $V_i$ of the form
    $\hat v_i, e\cdot \hat v_i,\ldots, e^{s-1}\cdot \hat v_i$.
  \item Set $g_s\cdot (e^k\hat v_i) = e^k\cdot (\hat g_s\cdot \hat v_i)$.
  \item Let $g_s$ be the identity on the $V_j$ with $\dim V_j\neq s$.  
  \end{enumerate}
\item Let $\mathcal{H}$ be the group generated by all $g_s$, for $s$ odd.
  Compute
  $\mathcal{H}_1 = \{ g\in \mathcal{H} \mid \det(g)=1\}$.
\item For $g\in \mathcal{H}_1$ compute the automorphism $\sigma_g\in
  \Aut(\o(V,\varphi))$
  given by $\sigma_g(x) = gxg^{-1}$.
\item Compute the centralizer $\z_\g(h,e,f)$ where $\g= \o(V,\varphi)$.
  By the algorithm indicated in Remark \ref{rem:outer} remove the $\sigma_g$
  that lie in the identity component of $Z_G(h,e,f)$ where $G$ is the
  identity component of $\Aut(\g)$.
\end{enumerate}

We have implemented this algorithm in the language of the computational
algebra system {\sf GAP}4. It works without particular problems. For example,
for the simple Lie algebra of type $B_{10}$ which is isomorphic to
$\o(21,\varphi)$ the program
needed 182 seconds to compute the component groups for all 195 nilpotent orbits.

\begin{example}
        Here we show how the algorithm works for a simple Lie algebra of type $B_2$. Let $V=\mathbb{C}^5$ and $\varphi_A(v,w)=\prescript{t}{}vAw$, where $v,w\in\mathbb{C}^5$ and $A$ is the $5\times 5$ matrix such that $A_{i,5-i+1}=1$ and $A_{i,j}=0$ elsewhere. Observe that $\varphi_A$ is symmetric, since $A$ is a symmetric matrix, so that $\widehat{G}=\mathcal{G}(\mathbb{C}^5,\varphi_A)=O(\mathbb{C}^5,\varphi_A)$ and its Lie algebra is $\ggg(\mathbb{C}^5,\varphi_A)=\mathfrak{so}(\mathbb{C}^5,\varphi_A)$, which is indeed of type $B_2$. We fix the following $\mathfrak{sl}_2$ triple:$$e=\begin{pmatrix}
        0 & 0 & 1 & 0 & 0\\
        0 & 0 & 0 & 0 & 0\\
        0 & 0 & 0 & 0 & -1\\
        0 & 0 & 0 & 0 & 0\\
        0 & 0 & 0 & 0 & 0\\
    \end{pmatrix},\hspace{0.3cm} h=\begin{pmatrix}
        2 & 0 & 0 & 0 & 0\\
        0 & 0 & 0 & 0 & 0\\
        0 & 0 & 0 & 0 & 0\\
        0 & 0 & 0 & 0 & 0\\
        0 & 0 & 0 & 0 & -2\\
    \end{pmatrix},\hspace{0.3cm} f=\begin{pmatrix}
        0 & 0 & 0 & 0 & 0\\
        0 & 0 & 0 & 0 & 0\\
        2 & 0 & 0 & 0 & 0\\
        0 & 0 & 0 & 0 & 0\\
        0 & 0 & -2 & 0 & 0\\
    \end{pmatrix}.$$ At this point we can easily compute the action of the nilpotent element $e$ on the canonical basis $\{v_1,\dots,v_5\}$ of $\mathbb{C}^5$: $ev_5=-v_3$, $ev_3=v_1$, $ev_1=0$, $ev_4=0$ and $ev_2=0$. Under the action of the triple, $\mathbb{C}^5$ decomposes as $\mathbb{C}\oplus\mathbb{C}\oplus\mathbb{C}^3$ and $v_2, v_4, v_5$ are the corresponding minimal vectors. We then have $M_1=\langle v_2, v_4\rangle$ and $M_3=\langle v_5\rangle$ and we fix the respective bases $\mathcal{B}_1=\{v_2,v_4\}$ and $\mathcal{B}_3=\{v_5\}$. At this point we need to compute the reflections: since $M_3$ is $1$-dimensional, $\hat{g}_3$ will just map $v_5\to -v_5$; for $\hat{g}_1$ we start observing that $\restr{\varphi_A}{M_1}(v_2,v_2)=\restr{\varphi_A}{M_1}(v_4,v_4)=0$. Then we set $w_1:=v_2-v_4$ and $w_2:=v_2+v_4$ and consider the map that sends $w_1\to w_1$, $w_2\to -w_2$. In this way, it is straightforward to get that $\hat{g}_3$ has to be the $2\times 2$ matrix with zero diagonal and $-1$ elsewhere (it is just the matrix associated to the endomorphism with respect to $\mathcal{B}_1$, it can be computed by basic linear algebra techniques). At this point we can use the isomorphism \eqref{fun}, \eqref{eq:prod}: \vspace{0.5cm}$$\begin{matrix}
        Z_{\widehat{G}}(h,e,f) & \xlongrightarrow{\sim} & O(M_1,\restr{\varphi_A}{M_1}) & \times & O(M_3,\restr{\varphi_A}{M_3})\\
        & & & &\\
        h_1:=\begin{pmatrix}
        1 & 0 & 0 & 0 & 0\\
        0 & 0 & 0 &-1 & 0\\
        0 & 0 & 1 & 0 & 0\\
        0 & -1 & 0 & 0 & 0\\
        0 & 0 & 0 & 0 & 1\\
        \end{pmatrix} & \testleftlong & \hat{g}_1=\begin{pmatrix}
            0 & -1\\
            -1 & 0
        \end{pmatrix} & \times & id_{M_3}=\begin{pmatrix}
            1
        \end{pmatrix}\\
        & & & &\\
                h_3:=\begin{pmatrix}
        -1 & 0 & 0 & 0 & 0\\
        0 & 1 & 0 & 0 & 0\\
        0 & 0 & -1 & 0 & 0\\
        0 & 0 & 0 & 1 & 0\\
        0 & 0 & 0 & 0 & -1\\
        \end{pmatrix} & \testleftlong & id_{M_1}=\begin{pmatrix}
            1 & 0\\
            0 & 1
        \end{pmatrix} & \times & \hat{g}_3=\begin{pmatrix}
            -1
        \end{pmatrix}\\
        
    \end{matrix}
$$
We get that $\langle h_1,h_3\rangle=\{id,h_1,h_3,h_1h_3\}\simeq C_2\times C_2$
is the component group of $Z_{\widehat{G}}(h,e,f)$.        
Following the final steps of our algorithm, it is then straightforward to see
that $\{id, h_1h_3\}\simeq C_2$ is the component group of $Z_G(h,e,f)$ (just observe that $h_1h_3$ is the only non trivial element with determinant $1$).
\end{example}

\begin{rmk}\label{rem:Cl}
Suppose that $\varphi$ is alternating, and hence the group we consider is
$\Sp(V,\varphi)$. Then everything works in the same way. The main difference
is that in this case $\psi_s$ is symmetric if $s$ is even and alternating if
$s$ is odd. So we get the same algorithm, except that the roles of $s$ even/odd
are interchanged.
\end{rmk}

\section{Conjugacy of $\ssl_2$-triples}\label{sec:conj}

In this section we let $\g$ be an arbitrary complex semisimple Lie algebra.
We use the notation introduced in Section \ref{sec:prelim}. In particular,
we let $\h$ denote a fixed Cartan subalgebra and denote the corresponding
root system by $\Phi$. We also fix a basis of simple roots $\Delta = \{
\alpha_1,\ldots,\alpha_\ell\}$. 

Throughout this section we 
let $(h_i,e_i,f_i)$, $i=1,2$, be two $\ssl_2$-triples in $\g$ where $e_1,e_2$ lie
in the same $G$-orbit. As remarked in Section \ref{sec:nilporb} this implies
that there exist $\sigma\in G$ with $\sigma(h_1)=h_2$, $\sigma(e_1)=e_2$,
$\sigma(f_1)=f_2$. In this section we describe computational methods to
find such a $\sigma$. We note that it
suffices to find $\sigma$ with $\sigma(h_1)=h_2$ and $\sigma(e_1)=e_2$ because
then automatically $\sigma(f_1)=f_2$ (\cite[Lemma 8.1.1]{gra16}).

First we assume that $h_1,h_2\in \h$.

The Weyl group $W$ of $\Phi$ acts on $\h$ in the following way.
For $\alpha\in \Phi$ there exists a unique $h_\alpha\in [\g_\alpha,\g_{-\alpha}]$
such that $[h_\alpha,x_{\pm\alpha}] = \pm 2x_{\pm\alpha}$.
Define $s_\alpha^\vee : \h\to\h$ by $s_\alpha^\vee(h) = h-\alpha(h) h_\alpha$.
Then the group generated by the $s_\alpha^\vee$ is isomorphic to $W$ (cf.
\cite[Remark 2.9.9]{gra16}). This action can also be realized differently.
Let $h_i$, $x_{\pm \alpha_i}$ for $1\leq i\leq\ell$ be a canonical generating set
of $\g$. For $t\in \C^\times$ define 
$w_{\alpha_i}(t) = \exp( t\ad x_{\alpha_i} ) \exp(-t^{-1}\ad x_{-\alpha_i})
\exp( t\ad x_{\alpha_i} )$ and set $\dot{s}_{\alpha_i}=w_{\alpha_i}(1)$. 
Then from \cite[Lemma 5.2.13]{gra16} it follows
that $\dot{s}_{\alpha_i}(h) = s_{\alpha_i}^\vee(h)$ for $h\in \h$. We also remark that
the action of $W$ preserves the space $\h_\Q$.

Let $h\in \h_\Q$ then there is a unique $\hat h$ in the $W$-orbit of $h$
such that $\alpha_i(h) \geq 0$ for $1\leq i\leq \ell$. (This is shown in the
same way as the corresponding statement for weights, see
\cite[Lemma 13.2A]{hum}). This element can be found in the following way. We define a sequence $h_1,\ldots,h_r$. Set $h_1=h$ and for $k\geq 1$ we do
the following. If $\alpha_i(h_k) \geq 0$ for all $i$ then we set $r=k$ and we
stop. Otherwise we set $h_{k+1} = s_{\alpha_i}^\vee(h_k)$ where $i$ is such that
$\alpha_i(h_k)<0$. We also set $k:= k+1$ and continue. A first observation is
that this sequence is always finite. Indeed,
we have a partial order on $\h_\Q$ defined by $u \leq u'$ if $u'-u =
\sum a_i h_i$ with $a_i\in \Q$ and $a_i\geq 0$ for all $i$. It follows
that $h_{k+1} >h_k$ for $k\geq 1$. Since the Weyl group is finite and we
compute an increasing series in the partial order, the sequence has to be
finite. It is clear that $\hat h=h_r$ is the element that we require.

Two elements of $\h$ are $G$-conjugate if and only if they are $W$-conjugate
(\cite[Corollary 5.8.5]{gra16}). So $h_1,h_2$ are $W$-conjugate. By the above
procedure we can compute $w_1,w_2\in W$ such that $\hat h_i = w_i(h)$.
By the uniqueness of the $\hat h_i$ it follows $\hat h_1 = \hat h_2$. 
Then $w_2^{-1}w_1(h_1)=h_2$. Because the action of the simple reflection
$s_{\alpha_i}$ is induced by $\dot{s}_{\alpha_i}$ and every $w\in W$ is a product of
simple reflections, we can thus find a $\tau\in G$ such that
$\tau(h_1)=h_2$. 

Set $h_1'=\tau(h_1)=h_2$, $e_1'=\tau(e_1)$, $f_1'=\tau(f_1)$. We want to find
an element in $G$ that stabilizes $h_2$ and maps $e_1'$ to $e_2$. For that
we consider the
group $Z_G(h_2) = \{ g\in G \mid g(h_2) = h_2\}$. This group is connected
by \cite[Corollary 3.11]{steinberg75}. Its Lie algebra is
$\z_\g(h_2) = \{ x\in \g \mid [x,h_2]=0\}$. Let $\Psi = \{\alpha\in \Phi
\mid \alpha(h_2) = 0\}$. Then $\Psi$ is a root subsystem of $\Phi$ and 
$$\z_\g(h_2) = \h \oplus \bigoplus_{\alpha\in\Psi} \g_\alpha.$$
In particular, $\z_\g(h_2)$ is reductive. Now we consider the Bruhat
decomposition of $Z_G(h_2)$. This is described as follows.
Let $W_0$ be the Weyl group of $\Psi$. Let $\Pi = \{\beta_1,\ldots,
\beta_r\}$ be a fixed basis of simple roots of $\Psi$. Let $x_{\pm \beta_i}$ be
the root vectors in a canonical generating set of (the semisimple part of)
$\z_\g(h_2)$. For $1\leq i\leq r$ define $\dot{s}_{\beta_i}$ as above. For
$w\in W_0$ let $w=s_{\beta_{i_1}}\cdots s_{\beta_{i_k}}$ be a fixed reduced
expression and define $\dot{w} = \dot{s}_{\beta_{i_1}}\cdots \dot{s}_{\beta_{i_k}}$.
Let $\{\beta_1,\ldots,\beta_m\}$ be the positive roots of $\Psi$
(corresponding to $\Pi$). Let $U$ be the subgroup of $Z_G(h_2)$ consisting of
all
$$\exp(s_1 \ad x_{\beta_1})\cdots \exp(s_m \ad x_{\beta_m}) \text{ for }
s_1,\ldots,s_m\in \C.$$
For $w\in W_0$ let $\Psi_w$ be the set of all $\beta_i$, $1\leq i\leq m$,
such that $w(\beta_i)$ is a negative root. Write $\Psi_w = \{\beta_{i_1},
\ldots,\beta_{i_n}\}$. Then $U_w$ is defined to be the subgroup of $Z_G(h_2)$
consisting of all
$$\exp(u_1 \ad x_{\beta_{i_1}})\cdots \exp(u_n \ad x_{\beta_{i_n}}) \text{ for }
u_1,\ldots,u_n\in \C.$$
For $1\leq i\leq \ell$ and $t\in \C^\times$ define $h_i(t) = w_{\alpha_i}(t)
w_{\alpha_i}(1)^{-1}$. Let $H = \{h_1(t_1)\cdots h_\ell(t_\ell) \mid
t_i\in \C^\times\}$. Then $H$ is a maximal torus in $Z_G(h_2)$ with Lie
algebra $\h$. We have that $Z_G(h_2)$ is the disjoint union of the sets
$$C_w= UH\dot{w} U_w$$
where $w$ runs over $W_0$ (see e.g.,
\cite[Proposition 5.1.10, Theorem 5.8.3]{gra16}). Now in order to find our
element of $Z_G(h_2)$ we run over $W_0$. For each $w\in W_0$ we write
$g\in C_w$ as
\begin{multline*}
g=\exp(s_1 \ad x_{\beta_1})\cdots \exp(s_m \ad x_{\beta_m})h_1(t_1)\cdots
h_\ell(t_\ell) \dot{w}\cdot\\
\exp(u_1 \ad x_{\beta_{i_1}})\cdots \exp(u_n \ad x_{\beta_{i_n}})
\end{multline*}
where we take the $s_i$, $u_j$, $t_k$ to be indeterminates of
a polynomial ring. Then the equation $g(e_1)=e_2$ is tantamount to
a set of polynomial equations in these indeterminates. We add auxiliary
indeterminates $a_1,\ldots,a_\ell$ and the polynomial equations $a_it_i=1$ for
$1\leq i\leq \ell$ that express the invertibility of the $t_i$.
With the technique of Gr\"obner bases we check whether these have a solution.
If they do, we
find one, for example by computing a Gr\"obner basis with respect to
a lexicographical order and solving the corresponding triangular system
(for an introduction into this technique we refer to \cite[\S 3.1]{clo}). 
For at least one $w$ a solution must exist because $e_1'$ and $e_2$ are
$Z_G(h_2)$-conjugate.

So by composing $\tau$ with the element previously found we obtain
a $\sigma\in G$ with $\sigma(h_1)=h_2$, $\sigma(e_1)=e_2$. But then
necessarily also $\sigma(f_1)=f_2$ by \cite[Lemma 8.1.1]{gra16}.

The second case occurs when $h_1,h_2$ do not lie in $\h$. Then we first find
Cartan subalgebras
$\h_i$ containing $h_i$. We compute the root systems of $\g$ with respect to
the $\h_i$, and corresponding canonical sets of generators. Mapping these
canonical generators to our fixed set of canonical generators yields
automorphisms of $\g$ mapping $\h_i\to \h$. By the algorithm indicated in
Remark \ref{rem:outer} we can check whether these are inner. If they
are not we compose with a diagram automorphism that stabilizes $\h$.
This way we obtain elements of $G$ mapping $h_i$ into $\h$. Subsequently
we continue as above.

\begin{rmk}\label{rem:0-dim}
Let $\z_\g(h,e,f) = \{x\in \g\mid [x,h]=[x,e]=[x,f]=0\}$. This is the Lie
algebra of $Z_G(h,e,f)$. So if $\z_\g(h,e,f)=0$ then $Z_G(h,e,f)$ is a finite
group, equalling the component group. In that case we can find the component
group by applying the above algorithm with $(h_i,e_i,f_i) = (h,e,f)$ for
$i=1,2$ and finding all solutions of the resulting polynomial equations.

We can also consider the group $A=\Aut(\g)$ and try to find the component
group of $Z_A(h,e,f)$. We have that $G=A^\circ$ so  $\z_\g(h,e,f)$ is also
the Lie algebra of $Z_A(h,e,f)$. Therefore if  $\z_\g(h,e,f)=0$ also
$Z_A(h,e,f)$ is a finite group. We can find it by an extension of the previous
method in the following way. By listing the diagram automorphisms of
$\g$ we can find $\sigma_1,\ldots,\sigma_r\in A$ such that the components
of $A$ precisely are $G\sigma_i$ for $1\leq i\leq r$. Fix an $i$ with
$1\leq i\leq r$. Write $h'=\sigma_i(h)$, $e'=\sigma_i(e)$, $f'=\sigma_i(f)$.
By computing weighted Dynkin diagrams (see \cite[\S 8.2]{gra16})
we can check whether $e,e'$ are $G$-conjugate. If they are not we discard
$\sigma_i$. Otherwise by the above methods we compute the set of all $\phi\in
G$ such that $\phi (h',e',f')=(h,e,f)$ (note that this is a finite set
because the stabilizer of such a triple is finite by our assumption). Then
$\phi\sigma_i$ where $\phi$ runs over the previously computed set, is the
set of all elements of $G\sigma_i$ stabilizing $(h,e,f)$. 
\end{rmk}

\begin{example}
Let $\g$ be the simple Lie algebra of type $F_4$ and $e\in \g$ a representative
of the nilpotent orbit with Bala-Carter label $F_4(a_3)$
(see \cite[\S 8.4]{colmcgov}). Let $(h,e,f)$ be an $\ssl_2$-triple. Then
$\z_\g(h,e,f)=0$. The centralizer $\z_\g(h)$ is of type $A_1+A_2+T_1$ (where
the $T_1$ indicates a 1-dimensional centre). So  $\z_\g(h)$ has four positive
roots and $|W_0|=12$. Hence the method of this section yields 12 systems of
polynomial equations in 16 indeterminates. It took {\sc Magma} 552 seconds
to compute the Gr\"obner bases of all systems with respect to a lexicographical
order. From these Gr\"obner bases the solutions are readily determined. 
There are 24 solutions and the stabilizer of $(h,e,f)$ is isomorphic to $S_4$.
\end{example}

\begin{example}
Let $\g$ be the simple Lie algebra of type $E_8$ and $e\in \g$ a representative
of the nilpotent orbit with Bala-Carter label $E_8(a_7)$. Let $(h,e,f)$ be an
$\ssl_2$-triple. Then
$\z_\g(h,e,f)=0$. The centralizer $\z_\g(h)$ is of type $A_3+A_4+T_1$ (where
the $T_1$ indicates a 1-dimensional centre). So  $\z_\g(h)$ has 16 positive
roots and $|W_0|=2880$. Hence the method of this section yields 2880 systems of
polynomial equations in 48 indeterminates. With two exceptions {\sc Magma}
could not compute the Gr\"obner bases of these polynomial systems. For this
reason we resort to different methods.

The first of these is based on the operations of factorization and reduction.
For a set $P$ of multivariate polynomials denote the set of their common
zeros in $\C$ by $Z(P)$. By the division algorithm (see \cite[\S 2.3]{clo})
it is straightforward to compute a set $P'$ such that no leading monomial
of an element of $P'$ divides a leading monomial of another element of $P'$
and such that $P$, $P'$ generate the same ideal. We say that $P'$ is a
{\em reduction} of $P$. We have that $Z(P)=Z(P')$.

Let $Q$ be a set of polynomials; we want to determine $Z(Q)$. 
Throughout our procedure we let $A$ be a set of sets of polynomials such that
$Z(Q)$ is the union of all $Z(Q')$ for $Q'\in A$. Let $P\in A$ and suppose
that there is an $f\in P$ with a nontrivial factorization $f=f_1^{e_1}\cdots
f_k^{e_k}$ with $f_i$ irreducible. Then we replace $P$ in $A$ with reductions
of the sets $P_i = (P\setminus \{f\})\cup \{f_i\}$. If an element of $A$
contains a nonzero constant then we discard it. We continue this process until
no further factorizations are possible. This way we can decide for many of our
2880 polynomial systems that they have no common zeros. More precisely,
after performing this procedure we are left with 507 sets of polynomials
to be checked. 

Another idea is to look for elements of order 2 (in this case the component group is
isomorphic to $S_5$ which is generated by elements of order 2).
We construct elements of the form
\begin{multline*}
A = \exp(s_1 \ad x_{\beta_1})\cdots \exp(s_m \ad x_{\beta_m})
h_1(t_1)\cdots h_\ell(t_\ell) \dot{w}\\
\exp(u_1 \ad x_{\beta_{i_1}})\cdots \exp(u_n \ad x_{\beta_{i_n}}).
\end{multline*}
The inverse of such an element is
\begin{multline*}
A^{-1} = \exp(-u_n \ad x_{\beta_{i_n}}) \cdots \exp(-u_1 \ad x_{\beta_{i_1}})
\dot{w}^{-1} h_1(t_1^{-1})\cdots h_\ell(t_\ell^{-1})\\
\exp(-s_m \ad x_{\beta_m}) \cdots \exp(-s_1 \ad x_{\beta_1}).
\end{multline*}
The condition that $A$ be of order 2 is the same as $A=A^{-1}$ yielding a set of
polynomial equations in the same indeterminates for each $w\in W_0$.
To these sets we applied the factorization procedure as outlined above. We then
added these sets to the original sets for the same $w\in W$.
Like this we found one polynomial set whose solutions yielded six elements of
order two. These elements generate a group of order 120. Since in this case the
stabilizer is known to be of order 120, we are done. 
\end{example}  

\section{Using double centralizers}\label{sec:double}

In this section we outline a computational method for determining the
component group of a $Z_G(h,e,f)$. For this we use the centralizer
$\c_1=\z_\g(h,e,f)$ and the double centralizer $\c_2=\z_\g(\c_1)$. The sum
$\c=\c_1+\c_2$ is reductive in $\g$. We write $\g = \c\oplus V$, where
$V$ is a $\c$-module (see below for its definition). The method is based on
the observation that for the restriction of $\sigma\in Z_G(h,e,f)$ to
$\c$ there are a finite number of possibilities (modulo $Z_G(h,e,f)^\circ$).
A second observation is that the $\c$-module $V$ is multiplicity free.
This forces a $\sigma$ to permute the simple summands of $V$. Finally, from
its restriction to $\c$ and the way it permutes the simple summands of $V$ we
can construct $\sigma$.

Throughout we let $\g$ be a simple Lie algebra of exceptional
type. Secondly, we assume that $e$ is such that the component group is
nontrivial. It is known for which nilpotent orbits that happens and 
those are the cases of interest. Section \ref{sec:tables} has tables
listing data on the algebras $\c_1$, $\c_2$ and the module $V$ used to
underpin the method described in this section. However, we remark that
this method also works for the cases where the component group is trivial,
see Remark \ref{triv}. 

Let $\kappa : \g\times\g\to \C$ denote the Killing form of $\g$, that is,
$\kappa(x,y) = \mathrm{Tr}(\ad x \cdot \ad y )$. It is a fundamental fact that
$\kappa$ is non-degenerate (\cite[Theorem 5.1]{hum}). 
We say that
a subalgebra $\uu\subset \g$ is {\em reductive in} $\g$ if $\g$ is completely
reducible as $\uu$-module. We start with a lemma containing some well known
facts on such subalgebras.

\begin{lemma}\label{lemred}
Let $\uu$ be reductive in $\g$.
\begin{itemize}
\item $\uu = [\uu,\uu]\oplus\z(\uu)$ where $\z(\uu)$ is the centre of $\uu$;
  moreover, for $x\in \z(\uu)$ the map $\ad x$ is semisimple.
\item Let $\z_\g(\uu) = \{ x\in \g \mid [x,y]=0 \text{ for all } y\in \uu\}$
  be the centralizer of $\uu$ in $\g$. Then the restriction of $\kappa$
  to $\z_\g(\uu)$ is non-degenerate and $\z_\g(\uu)$ is reductive in $\g$.
\end{itemize}  
\end{lemma}

\begin{proof}
For the decomposition of $\uu$ see \cite[Proposition 2.12.2]{gra16}. The
fact that $\ad x$ is semisimple for $x\in \z(\uu)$ follows from
\cite[Chapter III, Theorem 10]{jac}. The second part is shown in the same
way as the first part of the proof of \cite[Lemma 8.3.9]{gra16}.
\end{proof}  

Now fix a nilpotent $e\in \g$ lying in the $\ssl_2$-triple $(h,e,f)$.
Let $\c_1 = \z_\g(h,e,f) = \{x\in \g\mid [x,h]=[x,e]=[x,f]=0\}$. Here we
suppose that $\dim \c_1 >0$. (If $\dim \c_1=0$ then we can determine
$Z_G(h,e,f)$ as in Remark \ref{rem:0-dim}).
By the previous lemma $\c_1$ is reductive in $\g$. Set $\c_2 = \z_\g(\c_1)$.
Again by the previous lemma $\c_2$ is reductive in $\g$. For $i=1,2$ set
$\c_i' = [\c_i,\c_i]$ which is
the semisimple part of $\c_i$. By $\z(\c_i)$ we denote the centre of $\c_1$, so
that $\c_i = \c_i'\oplus \z(\c_i)$. 

\begin{lemma}\label{lem0}
We have $h,e,f\in \c_2$, $\z_\g(\c_2) = \c_1$ and $\c_1\cap\c_2 = \z(\c_1)
=\z(\c_2)$.   
\end{lemma}

\begin{proof}
The first statement is obvious. It is equally obvious that $\c_1\subset
\z_\g(\c_2)$. On the other hand, if $x\in \z_\g(\c_2)$ then $[x,h]=[x,e]=
[x,f]=0$ as $h,e,f\in \c_2$, implying $x\in \c_1$. We conclude that
$\z_\g(\c_2) = \c_1$. We have
$$\c_1\cap\c_2 = \{ x\in \c_1 \mid [x,\c_1]=0\} = \z(\c_1).$$
Similarly
$$\c_1\cap\c_2 = \{ x\in \c_2 \mid [x,\c_2]=0\} = \z(\c_2).$$
\end{proof}  

In the sequel we set $\ttt= \z(\c_1)=\z(c_2)$. The Lie algebra of the stabilizer
$Z_G(h,e,f) = \{ \sigma\in G \mid \sigma(h)=h, \sigma(e)=e, \sigma(f)=f\}$
is $\ad \c_1 = \{ \ad x \mid x\in \c_1\}$. Hence $\ad\ttt = \{\ad x\mid
x\in \ttt\}$ is the Lie algebra of a central torus $T$ of $Z_G(h,e,f)$. 

\begin{lemma}\label{lem1}
Let $\sigma\in Z_G(h,e,f)$. Then $\sigma$ stabilizes $\c_1'$, $\c_2'$ and
$\ttt$.   
\end{lemma}

\begin{proof}
Since the Lie algebra of $Z_G(h,e,f)$ is $\ad \c_1$ we have $\sigma (\ad \c_1)
\sigma^{-1} = \ad \c_1$. But $\sigma (\ad x ) \sigma^{-1} = \ad \sigma(x)$.
As the adjoint representation of $\g$ is faithful it follows for
$x\in \c_1$ that $\sigma(x)\in \c_1$. For $x\in \c_1$ and $y\in \c_2$ we
have $[x,\sigma(y)] = \sigma([\sigma^{-1}(x),y])=0$ as $\sigma^{-1}(x)\in \c_1$
by the first part. It follows that $\sigma(y)\in \c_2$. So $\sigma$ stabilizes
$\c_1$ and $\c_2$. Because $\sigma$ is an automorphism it stabilizes
$\c_i'$ and $\ttt=\z(\c_i)$ as well.
\end{proof}

Let $z\in \ttt$, $x,y\in \c_1$. Then $\kappa([x,y],z) = \kappa(x,[y,z])=0$
(for the first equality see \cite[\S 4.3]{hum}).
It follows that $\kappa(\c_1',\ttt)=0$. Since the restriction of $\kappa$ to
$\c_1$ is non-degenerate (Lemma \ref{lemred}) also its restrictions to
$\c_1'$, $\ttt$ are non-degenerate. Similarly we see that the restriction of
$\kappa$ to $\c_2'$ is non-degenerate. Set $\c = \c_1 + \c_2 =
\c_1'\oplus \c_2'\oplus \ttt$; then $\c$ is reductive in $\g$ and the
restriction of $\kappa$ to $\c$ is non-degenerate. 

Write $\c^\perp = \{ x\in \g\mid \kappa(x,y)=0 \text{ for all } y\in \c\}$.
As $\kappa$ is non-degenerate on $\c$ we have $\c\cap \c^\perp=0$ so that
$\g = \c\oplus \c^\perp$.

\begin{lemma}\label{lem2}
We have $[\c,\c^\perp]\subset \c^\perp$.
\end{lemma}

\begin{proof}
Let $x,y\in \c$ and $z\in\c^\perp$. Then $\kappa(x,[y,z]) = \kappa([x,y],z)=0$.
Hence $[y,z] \in \c^\perp$. 
\end{proof}  

Let $V=\c^\perp$ which is a $\c$-module by Lemma \ref{lem2}. For $x\in \c$ and
$v\in V$ we write $x\cdot v = [x,v]$.

The next lemma is proved by explicit computation, going through the list of
nilpotent orbits of the Lie algebras of exceptional type and for each
case calculating $V$ and its decomposition as a $\c$-module. We have performed
these computations in the computer algebra system {\sf GAP}4, using the
package {\sf SLA}. For the cases
where the stabilizer has a nontrivial component group we have listed the
highest weights of the irreducible summands in the tables in Section
\ref{sec:tables}.

\begin{lemma}\label{lem3}
Let $\g$ be of exceptional type. Then $V$ is multiplicity free.
\end{lemma}

Next we show ho to determine a finite set $\mathcal{S}$
of automorphisms of $\g$,
stabilizing $\c_1'$, $\ttt$, $\c_2'$ and such that each component of $Z_G(h,e,f)$
has an element lying in $\mathcal{S}$.

First consider $\c_1'$ and write $A_1=\Aut(\c_1')$.
This is a semisimple Lie algebra, so $A_1$ is described as in Section
\ref{sec:pre1}. In particular, we fix diagram automorphisms $\theta_1,\ldots,
\theta_p$ of $\c_1'$ such that $\theta_i A_1^\circ$ for $1\leq i\leq p$ are
the connected components of $A_1$.

\begin{lemma}\label{lem4}
Each component of $Z_G(h,e,f)$ has an element whose restriction to $\c_1'$ is
equal to one of the $\theta_i$.   
\end{lemma}

\begin{proof}
The identity component $A_1^\circ$ has Lie algebra
$\ad_{\c_1'} \c_1' = \{\ad_{\c_1'} x \mid x\in \c_1'\}$. The identity component
$Z_G(h,e,f)^\circ$ has Lie algebra $\ad \c_1 = \{ \ad x \mid x\in \c_1\}$.
Now $\ad_{\c_1'} \c_1'$ is isomorphic to the subalgebra $\ad \c_1' = \{ \ad x \mid
x\in \c_1'\}$ of $\ad \c_1$. Let $C_1'$ denote the connected algebraic subgroup
of $Z_G(h,e,f)^\circ$ whose Lie algebra is $\ad \c_1'$. We have a morphism
$C_1'\to
A_1^\circ$, mapping $\phi\in C_1'$ to its restriction to $\c_1'$ (note that
by Lemma \ref{lem1} $\phi$ stabilizes $\c_1'$). We have that $A_1^\circ$ is
generated by elements $\exp(\ad_{\c_1'} x)$ where $x\in \c_1'$ is nilpotent.
Such an element is the restriction of $\exp(\ad x) \in C_1'$. Hence the map
$C_1'\to A_1^\circ$ is surjective. Let $\sigma\in Z_G(h,e,f)$ then
the restriction $\sigma'$ of $\sigma$ to $\c_1'$ lies in $A_1$. Hence there is a
$\phi'\in A_1^\circ$ such that $\sigma'\phi'$ is equal to $\theta_i$ for an $i$
with $1\leq i\leq p$. Let $\phi\in C_1'$ map to $\phi'$. Then the restriction
of $\sigma\phi$ to $\c_1'$ is equal to $\theta_i$. Furthermore, $\sigma\phi$
lies in the same component of $Z_G(h,e,f)$ as $\sigma$. 
\end{proof}

We want to compute one element of each component
of $Z_G(h,e,f)$. Therefore we may assume that the restriction of such an
element to $\c_1'$ is equal to one of the $\theta_i$. 

Note that $h,e,f\in \c_2'$. Furthermore, $\z_\g(h,e,f) = \c_1$, so that
$\z_{\c_2'}(h,e,f)=0$. Let $A_2= \Aut(\c_2')$. Let $Z_{A_2}(h,e,f)$ denote the
stabilizer of $h,e,f$ in $A_2$. The Lie algebra of $Z_{A_2}(h,e,f)$ is
$\{ \ad_{\c_2'} x \mid x \in \z_\g(h,e,f) \cap \c_2'\}$. Hence this Lie algebra
is trivial, and it follows that $Z_{A_2}(h,e,f)$ is finite. Therefore we can
determine $Z_{A_2}(h,e,f)$ as seen in Remark \ref{rem:0-dim}. Since the
restriction of an element of $Z_G(h,e,f)$ to $\c_2'$ lies in $Z_{A_2}(h,e,f)$
we can determine a finite set of automorphisms $\eta_1,\ldots,\eta_q$ 
of $\c_2'$ such that the
restriction of a $\sigma\in Z_G(h,e,f)$ to $\c_2'$ is equal to one of the
$\eta_i$.

Let $h_1,\ldots,h_s$, $x_1,\ldots,x_s$, $y_1,\ldots,y_s$ be a canonical
generating set of the semisimple subalgebra $\c'=\c_1'\oplus \c_2'$.
Here the $h_i$ span a fixed Cartan subalgebra of $\c'$, the
$x_i$ are root vectors of $\c'$ corresponding to the simple positive roots, and the
$y_i$ are root vectors corresponding to the simple negative roots. As seen in
Section \ref{sec:red} we can compute a direct sum decomposition $V = V_1\oplus
\cdots\oplus V_m$, where each $V_i$ is irreducible. Moreover, because of
Lemma \ref{lem3} this decomposition is uniquely determined. For each $V_i$
we fix a highest weight vector $v_i$. 

Write $d=\dim \ttt$ (from the tables in Section \ref{sec:tables} we
see that $d\leq 2$ if the component group of $Z_G(h,e,f)$ is nontrivial).
Fix a basis $h_{s+1},\ldots,h_{s+d}$ of $\ttt$. From the tables in Section
\ref{sec:tables} we can read off the values $\nu_{ij}$ with $h_i\cdot v_j =
\nu_{ij} v_j$, $1\leq i\leq s+d$, $1\leq j\leq m$.

Fix a $\theta_u$ and an $\eta_v$. Next we study properties of the $\sigma\in
G$ whose restrictions to $\c_1'$, $\c_2'$ are $\theta_u$, $\eta_v$ respectively.
We say that such a $\sigma$ is a {\em $(u,v)$-extension}. We remark that for
given $u,v$ the set of all $(u,v)$-extensions could be empty. First we show
that we can determine the restriction to $\ttt$ of any $(u,v)$-extension.
Note that any $(u,v)$-extension automatically lies in $Z_G(h,e,f)$.

\begin{rmk}\label{rem:vext}
It may happen that $\dim\c_1'=0$ (for an example see Example \ref{exa:E6}).
In that case there is no $\theta_u$. However, everything that follows works
unchanged and instead of $(u,v)$-extensions we use the term $(v)$-extensions. 
\end{rmk}  

Fix $u$ and $v$ and let $\sigma\in G$ be a $(u,v)$-extension. 
Set $\bar h_i = \sigma(h_i)$, $\bar x_i = \sigma(x_i)$
and $\bar y_i = \sigma(y_i)$. Except $\bar h_i$ for $s+1\leq i
\leq s+d$ these elements can be computed from the knowledge of $u,v$ because
we know the restrictions of $\sigma$ to $\c_1'$, $\c_2'$. 
For $1\leq j\leq m$ let $\bar v_j$ be a nonzero element of
$V_j$ with $\bar x_i \cdot \bar v_j = 0$ for $1\leq i\leq s$. Then $\bar v_j$
is a highest weight vector of $V_j$ with respect to the new canonical
generating set consisting of the $\bar h_i$, $\bar x_i$, $\bar y_i$.
Hence, up to scalar multiples, it is uniquely determined. 
Since the decomposition of $V$ is unique, we must have that $\sigma(v_j)$ is
a nonzero scalar multiple of some $\bar v_l$. We have $h_i \cdot v_j =
\nu_{ij} v_j$ for $1\leq i\leq s$. Applying $\sigma$ we see that $\bar h_i \cdot
\sigma(v_j) = \nu_{ij} \sigma(v_j)$ for $1\leq i\leq s$. Let $P_{uv}$ be the
set of permutations $\pi$ 
of $\{1,\ldots,m\}$ such that $\bar h_i \cdot \bar v_{\pi(j)} = \nu_{ij}
\bar v_{\pi(j)}$ for $1\leq i\leq s$ and $1\leq j\leq m$.
There can be more than one such permutation, but if there
is no such permutation then we can immediately conclude that there are
no $(u,v)$-extensions. On the other hand if there is a $(u,v)$-extension
$\sigma$ then $\sigma(v_j) = \lambda_j \bar v_{\pi(j)}$ for
some $\pi\in P_{uv}$ and nonzero $\lambda_j\in \C$.

Now suppose that $P_{uv}\neq\emptyset$ and let $\pi\in P_{uv}$. Let $\sigma$ be
a $(u,v)$-extension with $\sigma(v_j) = \lambda_j \bar v_{\pi(j)}$
for $1\leq j\leq m$. We call
such a $\sigma$ a {\em $(u,v,\pi)$-extension}. We write $\bar h_{s+i}=
\sigma(h_{s+i}) = \sum_{k=1}^d a_{ik} h_{s+k}$. The elements of $\ttt$ are
central in $\c$ and therefore they act as scalar multiplication on irreducible
$\c$-modules (this is Schur's lemma, cf. \cite[\S 6.1]{hum}). It follows that
$h_{s+i}\cdot \bar v_j = \nu_{s+i,j} \bar v_j$ for $1\leq i\leq d$,
$1\leq j\leq m$. Now  $h_{s+i}\cdot v_j=
\nu_{s+i,j} v_j$ upon application of $\sigma$ implies that $\bar h_{s+i} \cdot
\bar v_j = \nu_{s+i,\pi^{-1}(j)} \bar v_j$ for $1\leq j\leq m$.
We see that $\sum_{k=1}^d a_{ik}\nu_{s+k,j} = \nu_{s+i,\pi^{-1}(j)}$ for $1\leq i\leq
d$ and $1\leq j\leq m$. The values of the $\nu_{s+i,j}$ can be read off from
the tables in Section \ref{sec:tables}. From that we see that the vectors
$(\nu_{s+1,j},\ldots,\nu_{s+d,j})$ for $1\leq j'\leq m$ span $\C^d$.
Hence either the equations
give a unique solution for the $a_{ik}$, or no solution. In the latter case we
conclude that there are no $(u,v,\pi)$-extensions. In the former case
we can uniquely compute the restriction of any $(u,v,\pi)$-extension to $\ttt$.

Let $\sigma$ be a $(u,v,\pi)$-extension. Let $\phi\in T$; then the restriction
of $\phi$ to $\c$ is the identity. Because $V$ is multiplicity free we have that
$\phi(v_j)$ is a scalar multiple of $v_j$ for $1\leq j\leq m$. Hence
$\sigma\phi$ also is a $(u,v,\pi)$-extension. We see that if there are
$(u,v,\pi)$-extensions and $\dim T\geq 1$ then there are infinitely many
such extensions. Our next goal is to determine a finite set of
$(u,v,\pi)$-extensions such that every $(u,v,\pi)$-extension is equal, modulo
$Z_G(h,e,f)^\circ$, to an element of this set. 

Now we consider the group
$$H = \{ \sigma\in G \mid \sigma(x)=x \text{ for all } x \in \c\}.$$
This is an algebraic subgroup of $G$ with Lie algebra $\ad\z_\g(\c)$. But the
latter is $\ad\ttt$ by Lemma \ref{lem0}. It follows that $H^\circ = T$.
If $\sigma\in H$ then $\sigma(v_i)$ is a highest weight vector with the same
weight as $v_i$. So because $V$ is multiplicity free it follows that
$\sigma(v_i) = \chi_i(\sigma)v_i$, where $\chi_i : H\to \C^*$ is a character
of $H$. The differential $\mathrm{d} \chi_i : \ttt\to \C$ is a linear map;
it can be read off from the tables in Section \ref{sec:tables}. From those
tables it is also easily seen that in all cases there are $d$ of such
characters $\chi_{i_1},\ldots, \chi_{i_d}$ such that the intersection of
the kernels of the $\mathrm{d} \chi_{i_l}$, $1\leq l\leq d$ is 0 (if $d=1$ this
just means that there is an $i_1$ such that $\mathrm{d} \chi_{i_1}$ is nonzero).
This implies
that $K=\{ \sigma\in H \mid \chi_{i_l}(h)=1 \text{ for } 1\leq l\leq d\}$ is a
finite group.

\begin{lemma}\label{lem5}
Let $\sigma$ be a $(u,v,\pi)$-extension. Then there is a $\phi\in T$ such that
$\sigma\phi(v_{i_l})= \bar v_{\pi(i_l)}$ for $1\leq l\leq d$. 
\end{lemma}  

\begin{proof}
Define $\xi : T\to (\C^*)^d$ by $\xi(h) = (\chi_{i_1}(h),\ldots,\chi_{i_d}(h))$.
Since the intersections of the kernels of the $\mathrm{d}\chi_i$ is zero,
we have that $\mathrm{d}\xi : \ttt\to \C^d$ is bijective. It follows that
$\xi(T)$ is an algebraic subgroup of $(\C^*)^d$ whose Lie algebra is
$\C^d$. In particular the dimension of $\xi(T)$ is $d$, which implies that
$\xi(T) = (\C^*)^d$.

We have that $\sigma(v_j) = \lambda_j \bar v_{\pi(j)}$ for $1\leq j\leq m$,
where the $\lambda_j\in \C$ are nonzero. 
Let $\phi\in T$ be such that $\xi(\phi) = (\lambda_{i_1}^{-1},\ldots,
\lambda_{i_d}^{-1})$. Then $\sigma\phi$ has the desired property.
\end{proof}

Let $\sigma_1,\sigma_2\in G$ be two $(u,v,\pi)$-extensions and write
$\sigma_i(v_j) = \lambda_j^i \bar v_{\pi(j)}$ for $1\leq j\leq m$.  
Suppose that $\lambda_{i_l}^r=1$ for $1\leq l\leq d$ and $r=1,2$. Then
$\sigma_2^{-1}\sigma_1$ lies in $H$ and $\chi_{i_l}(\sigma_2^{-1}\sigma_1) = 1$ for
$1\leq l\leq d$. Hence $\sigma_2^{-1}\sigma_1$ lies in $K$ which is a finite
group. We conclude that there is a finite number of $(u,v,\pi)$-extensions
$\sigma$ with $\sigma(v_{i_l}) = \bar v_{\pi(i_l)}$ for $1\leq l\leq d$. 

We remark that the sets
\begin{align*}
  & \{h_1,\ldots,h_{s+d},x_1,\ldots,x_s,y_1,\ldots,y_s,v_1,\ldots,v_m\}\\
  & \{\bar h_1,\ldots,\bar h_{s+d},\bar x_1,\ldots,\bar x_s,\bar y_1,\ldots,
  \bar y_s,\bar v_1,\ldots,\bar v_m\}\\
\end{align*}
generate $\g$. So a $\sigma\in G$ with $\sigma(h_i) = \bar h_1$, $\sigma (x_i) =
\bar x_i$, $\sigma(y_i) = \bar y_i$, $\sigma(v_j) = \lambda_j \bar v_j$ is
uniquely determined. Furthermore, we can compute the matrix of $\sigma$
with respect to a given basis of $\g$. This matrix has entries that are
polynomials in the $\lambda_j$. We compute those polynomials. Then the
requirement that $\sigma$ be an automorphism yields polynomial equations
in the $\lambda_j$. We compute these equations, where we now view the
$\lambda_i$ as indeterminates. Furthermore we set $\lambda_j=1$ (that is, we
add the equations $\lambda_j-1=0$) for $j\in \{i_1,\ldots,i_d\}$). Finally we
add equations of the form $\lambda_i \hat\lambda_i=1$, where the
$\hat \lambda_j$ are auxiliary indeterminates (these equations make sure
that the value of $\lambda_j$ is nonzero). As seen above these equations
have a finite number of solutions. We determine these solutions, using the
technique based on computing a Gr\"obner basis. Doing this for all $u,v,\pi$
we obtain a finite set of automorphisms of $G$ such that every component
has one element in the set. By the method referred to in Remark \ref{rem:outer}
we get rid of pairs of automorphisms that are equal modulo $Z_G(h,e,f)^\circ$.

\begin{rmk}\label{triv}
Let $e$ be a nilpotent element of $\g$ lying in an $\ssl_2$-triple
$(h,e,f)$ such that the component group of $Z_G(h,e,f)$ is trivial. We consider
the subalgebras $\c_1'$, $\ttt$, $\c_2'$. By direct computation we have verified
the following statements. If $\g$ is of type $G_2$ or $F_4$ then
$\dim \ttt = 0$. If $\g$ is of type $E_6$ then $\dim \ttt\leq 1$. Furthermore,
in all cases where $\dim \ttt=1$ we can find a $\chi_{i_1}$ such that
$\mathrm{d}\chi_{i_1}$ is nonzero (notation as above). If $\g$ is of type
$E_7$ or $E_8$ then $\dim \ttt=0$. This means that the method outlined in
this section also works when the component group is trivial. 
\end{rmk}

\begin{example}\label{exa:E6}
Here we look at the simple Lie algebra $\g$ of type $E_6$ and the nilpotent
orbit with label $D_4(a_1)$. In this case $\c_1=\ttt$ which is of dimension 2.
Hence $\c_1'=0$. Furthermore $\c_2'$ is semisimple of type $D_4$. For the
Dynkin diagram of $\c_2'$ we use the following enumeration of the vertices:
\begin{center}
\begin{tikzpicture}
\node[] (1) at (0,0) {1};
\node[] (2) at (1,0) {2};
\node[] (3) at (2,0) {3};
\node[] (4) at (1,1) {4};
\path (1) edge[sedge] (2)
(2) edge[sedge] (4)
(2) edge[sedge] (3);
\end{tikzpicture}
\end{center}

In {\sf GAP}4 we have explicitly determined the $\c$-module $V$ and its
decomposition. It turns out that $V$ splits as a direct sum of six irreducible
modules with highest weights
\begin{center}
\begin{tikzpicture}
\node[] (11) at (0,0) {$V_1:$};   
\node[] (1) at (1,0) {1};
\node[] (2) at (2,0) {0};
\node[] (3) at (3,0) {0};
\node[] (4) at (2,1) {0};
\node[] (5) at (4,0) {(-2,0)};
\path (1) edge[sedge] (2)
(2) edge[sedge] (4)
(2) edge[sedge] (3);

\node[] (12) at (6,0) {$V_2:$};
\node[] (6) at (7,0) {1};
\node[] (7) at (8,0) {0};
\node[] (8) at (9,0) {0};
\node[] (9) at (8,1) {0};
\node[] (10) at (10,0) {(2,0)};
\path (6) edge[sedge] (7)
(7) edge[sedge] (8)
(7) edge[sedge] (9);
\end{tikzpicture}

\begin{tikzpicture}
\node[] (11) at (0,0) {$V_3:$};   
\node[] (1) at (1,0) {0};
\node[] (2) at (2,0) {0};
\node[] (3) at (3,0) {0};
\node[] (4) at (2,1) {1};
\node[] (5) at (4,0) {(-1,-3)};
\path (1) edge[sedge] (2)
(2) edge[sedge] (4)
(2) edge[sedge] (3);

\node[] (12) at (6,0) {$V_4:$};
\node[] (6) at (7,0) {0};
\node[] (7) at (8,0) {0};
\node[] (8) at (9,0) {0};
\node[] (9) at (8,1) {1};
\node[] (10) at (10,0) {(1,3)};
\path (6) edge[sedge] (7)
(7) edge[sedge] (8)
(7) edge[sedge] (9);
\end{tikzpicture}

\begin{tikzpicture}
\node[] (11) at (0,0) {$V_5:$};   
\node[] (1) at (1,0) {0};
\node[] (2) at (2,0) {0};
\node[] (3) at (3,0) {1};
\node[] (4) at (2,1) {0};
\node[] (5) at (4,0) {(-1,3)};
\path (1) edge[sedge] (2)
(2) edge[sedge] (4)
(2) edge[sedge] (3);

\node[] (12) at (6,0) {$V_6:$};
\node[] (6) at (7,0) {0};
\node[] (7) at (8,0) {0};
\node[] (8) at (9,0) {1};
\node[] (9) at (8,1) {0};
\node[] (10) at (10,0) {(1,-3)};
\path (6) edge[sedge] (7)
(7) edge[sedge] (8)
(7) edge[sedge] (9);
\end{tikzpicture}
\end{center}

Here the eigenvalues of the Cartan elements of $\c_2'$ are shown on the
Dynkin diagram of $\c_2'$. This is followed by a vector of length 2 containing
the eigenvalues of the elements of a fixed basis of $\ttt$ on the highest
weight vector. 

With the method indicated in Remark \ref{rem:0-dim}
it is straightforward to compute $Z_{A_2}(h,e,f)$ where $A_2=\Aut(\c_2')$.
The symmetry group of the Dynkin diagram is exactly the group $S_X$ permuting
the vertices in $X=\{1,3,4\}$. For each diagram automorphism $\sigma_\tau$
of $\c_2'$, with $\tau\in S_X$, there is exactly one inner automorpism
$\phi$ such that $\phi (\sigma_\tau(h),\sigma_\tau(e),\sigma_\tau(f))=(h,e,f)$.
Hence $Z_{A_2}(h,e,f)$ is of order 6 and is isomorphic to $S_X$.

We consider $\tau=(1,3,4)\in S_X$ and let $\eta_1$ be the corresponding
element of $Z_{A_2}(h,e,f)$. We want to determine the possible extensions of
$\eta_1$ to automorphisms of $\g$. That is, we want to determine the
$(1)$-extensions (cf. Remark \ref{rem:vext}). 
We fix canonical generators $h_i$, $x_i$, $y_i$
$1\leq i\leq 4$ of $\c_2'$ and let $\bar h_i = \eta_1(h_i)$,  $\bar x_i =
\eta_1(x_i)$,  $\bar y_i = \eta_1(y_i)$. Let $v_j$ for $1\leq j\leq 6$ denote
a fixed highest weight vector of $V_j$ and let $\bar v_j$ be a nonzero element
of $V_j$ such that $\bar x_i \cdot \bar v_j=0$ for $1\leq i\leq 4$. If
$h_i \cdot v_j = \nu_{i,j} v_j$ then $\bar h_i \cdot \bar v_j = \nu_{\tau^{-1}(i),j}
\bar v_j$. It follows that if $\sigma\in G$ extends $\eta_1$ then
$\sigma(v_j) = \lambda_i \bar v_{\pi(j)}$ where $\pi$ is a permutation of
$\{1,\ldots,6\}$ mapping
\begin{align*}
& 1\mapsto 5 \text{ or } 6\\
& 2\mapsto 5 \text{ or } 6\\
& 3\mapsto 1 \text{ or } 2\\
& 4\mapsto 1 \text{ or } 2\\
& 5\mapsto 3 \text{ or } 4\\
& 6\mapsto 3 \text{ or } 4\\
\end{align*}
So we have eight possible permutations.

Let $\sigma\in G$ extend $\tau$ and suppose that $\sigma(v_1) = \lambda_1
\bar v_5$, $\sigma(v_3) = \lambda_3 \bar v_1$ and $\sigma(v_5) =
\lambda_5 \bar v_4$. Let $h_5,h_6$ be
the basis elements of $\ttt$ with respect to which the above weights have been
computed. As remarked above, Schur's lemma implies that $h_5$, $h_6$ have
the same eigenvalues on $\bar v_i$ as on $v_i$.
Write $\sigma(h_5) = ah_5+bh_6$, $\sigma(h_6) = ch_5+dh_6$.
Applying $\sigma$ to the equality $h_5\cdot v_1 = -2v_1$ we see that
$(ah_5+bh_6)\cdot \bar v_5 =-2\bar v_5$, implying $-a+3b=-2$. Applying
$\sigma$ to $h_5\cdot v_3 = -v_3$ we obtain $-2a=-1$ and we see that
$a=\tfrac{1}{2}$, $b=-\tfrac{1}{2}$. But applying $\sigma$ to $h_5\cdot v_5=
-v_5$ we get $-a-3b=-1$ and we have a contradiction. The conclusion is that
this permutation does not work.

We apply the above procedure to all eight permutations. It turns out that only
one of them yields a solution for $a,b,c,d$ namely
$$\pi:~ 1\mapsto 6,\, 2\mapsto 5,\, 3\mapsto 2,\, 4\mapsto 1,\, 5\mapsto 3,\,
6\mapsto 4.$$
In this case $a=-\tfrac{1}{2}$, $b=\tfrac{1}{2}$, $c=-\tfrac{3}{2}$,
$d=-\tfrac{1}{2}$.

So we try to find $\sigma\in G$ with $\sigma(h_i)=\bar h_i$,
$\sigma(x_i)=\bar x_i$, $\sigma(y_i) = \bar y_i$ for $1\leq i\leq 4$,
$\sigma(h_5) = ah_5+bh_6$, $\sigma(h_6) = ch_5+dh_6$ (with $a,b,c,d$ as above)
and $\sigma(v_j) = \lambda_j \bar v_{\pi(j)}$ with $\pi$ as above. In other
words, we try to find all $(1,\pi)$-extensions. 

We have written an ad-hoc program in {\sf GAP}4 for computing the polynomial
equations that the $\lambda_i$ have to satisfy in order that $\sigma$ be
an automorphism. It took the program somewhat more than 13 seconds to compute
these equations. The ensuing Gr\"obner basis computation terminated almost
immediately. Setting $\lambda_1= \lambda_3=1$ the equations yield exactly
one solution.

Doing this for all elements of $Z_{A_2}(h,e,f)$ we obtain exactly six elements
of $Z_G(h,e,f)$ and it follows that the component group is isomorphic to
$S_X$. 
\end{example}

\section{Tables for the exceptional types}\label{sec:tables}

In this section we give some data relative to the nilpotent orbits of the
Lie algebras of exceptional type, whose stabilizers have nontrivial
component group (Tables \ref{tab:G2}, \ref{tab:F4}, \ref{tab:E6}, \ref{tab:E7},
\ref{tab:E8}). The first column of these tables displays the Bala Carter
label of the orbit (see \cite[Chapter 8]{colmcgov}).
The second and fourth columns have
respectively the isomorphism type of the subalgebras $\c_1'$ and $\c_2'$.
Here a notation like $2A_2$ means a direct sum of two subalgebras, both
of type $A_2$. The third column has the dimension of the centre $\ttt$.
The fifth column has the highest weights of the $\c$-module $V$. Here the
notation works as follows. The canonical generators of $\c_1'$ are
$h_i,x_i,y_i$ with $1\leq i\leq a$. The basis elements of $\ttt$ are
$h_i$ with $a+1\leq i\leq a+d$. The canonical generators of $\c_2'$ are
$h_i,x_i,y_i$ with $a+d+1\leq i\leq r$. The weights are written as 
$(w_1 ; w_2; w_3 )$ where $w_1$ has the eigenvalues of the $h_i$, $1\leq i\leq a$
on the highest weight vector, $w_2$ (written in boldface) has the eigenvalues
of the $h_i$, $a+1\leq i\leq a+d$ on the highest weight vector, $w_3$ has the
eigenvalues of the $h_i$, $a+d+1\leq i\leq r$ on the highest weight vector.
If $\dim \c_1'=0$ or $\dim \ttt=0$ then we just write $(w_2;w_3)$, respectively
$(w_1;w_3)$. The eigenvalues of direct summands of $\c_1'$, $\c_2'$ are
separated by commas. For the enumeration of the nodes of the Dynkin diagrams
of the simple Lie algebra we follow the conventions in
\cite[Planche I-IX]{bou4}.
All weights are of multiplicity one (i.e.,
each weight belongs to a unique irreducible summand of $V$).
The last column has the isomorphism type of the
component group of the stabilizer (which is contained in the tables of
\cite[Chapter 8]{colmcgov}, \cite{sommers}). The data in the second, third,
fourth and fifth
columns have been computed with {\sf GAP}, with the help of the
{\sf SLA} package.

\begin{longtable}{|l|l|l|l|l|}
\caption{Nilpotent orbits in the Lie algebra of type $G_2$ whose stabilizer
  has a nontrivial component group}\label{tab:G2} 
\endfirsthead
\hline
\endhead
\hline
\endfoot
\endlastfoot

\hline
label & $\c_1'$ & $\dim \ttt$ & $\c_2'$ & $A$\\
\hline
$G_2(a_1)$ & 0 & 0 & $\g$ & $S_3$\\
\hline  
\end{longtable}

\begin{longtable}{|l|l|l|l|l|l|}
\caption{Nilpotent orbits in the Lie algebra of type $F_4$ whose stabilizer
  has a nontrivial component group}\label{tab:F4} 
\endfirsthead
\hline
\endhead
\hline
\endfoot
\endlastfoot

\hline
label & $\c_1'$ & $\dim \ttt$ & $\c_2'$ & weights of $V$ & $A$\\
\hline
$\widetilde{A}_1$ & $A_3$ & 0 & $A_1$ & $\wts{100}{1}$, $\wts{001}{1}$,
$\wts{010}{2}$ & $S_2$\\
$A_2$ & $A_2$ & 0 & $A_2$& $\wts{20}{01}$, $\wts{02}{10}$ & $S_2$\\
$B_2$ & $2A_1$ & 0 & $B_2$ & $\wts{1,1}{10}$, $\wts{1,0}{01}$, $\wts{0,1}{01}$ &
$S_2$\\
$C_3(a_1)$ & $A_1$ & 0 & $C_3$ & $\wts{1}{001}$ & $S_2$\\
$F_4(a_3)$ & 0 & 0 & $\g$ & & $S_4$ \\
$F_4(a_2)$ & 0 & 0 & $\g$ & &  $S_2$\\
$F_4(a_1)$ & 0 & 0 & $\g$ & &  $S_2$\\
\hline  
\end{longtable}  

\begin{longtable}{|l|l|l|l|l|l|}
\caption{Nilpotent orbits in the Lie algebra of type $E_6$ whose stabilizer
  has a nontrivial component group}\label{tab:E6}  
\endfirsthead
\hline
\endhead
\hline
\endfoot
\endlastfoot

\hline
label & $\c_1'$ & $\dim\ttt$ & $\c_2'$ & weights of $V$ & $A$ \\
\hline
$A_2$ & $2A_2$ & 0 & $A_2$ & $\wts{10,01}{01}$, $\wts{01,10}{10}$ & $S_2$\\
$D_4(a_1)$ & 0 & 2 & $D_4$ & $\wts{{\mathbf{-2,0}}}{1000}$,
$\wts{{\mathbf{2,0}}}{1000}$ &  $S_3$\\
& & & &  $\wts{{\mathbf{-1,-3}}}{0010}$,  $\wts{{\mathbf{1,3}}}{0010}$ & \\
& & & &   $\wts{{\mathbf{-1,3}}}{0001}$,  $\wts{{\mathbf{1,-3}}}{0001}$ &\\
$E_6(a_3)$ & 0 & 0 & $\g$ & & $S_2$ \\
\hline  
\end{longtable}  

\begin{longtable}{|l|l|l|l|l|l|}
\caption{Nilpotent orbits in the Lie algebra of type $E_7$ whose stabilizer
  has a nontrivial component group}\label{tab:E7} 
\endfirsthead
\hline
\endhead
\hline
\endfoot
\endlastfoot

\hline
label & $\c_1'$ & $\dim\ttt$ & $\c_2'$ & weights of $V$ & $A$\\
\hline
$A_2$ & $A_5$ & 0 & $A_2$ & $\wts{00010}{10}$, $\wts{01000}{01}$ & $S_2$\\
$A_2+A_1$ & $A_3$ & 1 & $A_1+A_2$ & $\wtt{000}{4}{0,01}$, $\wtt{000}{-4}{0,10}$ &
$S_2$\\
& & & & $\wtt{010}{-2}{0,01}$, $\wtt{010}{2}{0,10}$ & \\
& & & & $\wtt{001}{-3}{1,00}$, $\wtt{100}{3}{1,00}$ & \\
& & & & $\wtt{001}{1}{1,01}$, $\wtt{100}{-1}{1,10}$ & \\
$D_4(a_1)$ & $3A_1$ & 0 & $D_4$ & $\wts{1,1,0}{1000}$, $\wts{1,0,1}{0010}$,
$\wts{0,1,1}{0001}$ & $S_3$\\
$D_4(a_1)+A_1$ & $2A_1$ & 0 & $A_1+D_4$ & $\wts{1,1}{0,1000}$,
$\wts{1,0}{1,0001}$, $\wts{0,1}{1,0010}$ & $S_2$ \\
$A_3+A_2$ & $A_1$ & 1 &  $A_2+A_3$ & $\wtt{0}{-4}{01,000}$, $\wtt{0}{4}{10,000}$
& $S_2$\\
& & & & $\wtt{1}{-3}{00,001}$, $\wtt{1}{3}{00,100}$ & \\
& & & & $\wtt{0}{2}{01,010}$, $\wtt{0}{-2}{10,010}$ & \\
& & & & $\wtt{1}{-1}{01,100}$, $\wtt{1}{1}{10,001}$ & \\
$A_4$ & $A_2$ & 1 & $A_4$ & $\wtt{00}{-3}{0001}$, $\wtt{00}{3}{1000}$ & $S_2$\\
& & & & $\wtt{01}{2}{0001}$, $\wtt{10}{-2}{1000}$ & \\
& & & & $\wtt{01}{-1}{0010}$, $\wtt{10}{2}{0100}$ & \\
$A_4+A_1$ & 0 & $2$ & $A_1+A_4$ & $\wts{{\mathbf{3,-1}}}{1,0000}$,  $\wts{{\mathbf{-3,1}}}{1,0000}$ &  $S_2$\\
& & & &  $\wts{{\mathbf{0,-2}}}{0,1000}$,  $\wts{{\mathbf{0,2}}}{0,0001}$ &\\
& & & &  $\wts{{\mathbf{-2,2}}}{0,1000}$,  $\wts{{\mathbf{2,-2}}}{0,0001}$ &\\
& & & &  $\wts{{\mathbf{1,1}}}{1,1000}$,  $\wts{{\mathbf{-1,-1}}}{1,0001}$ &\\
& & & &  $\wts{{\mathbf{-2,0}}}{0,0100}$,  $\wts{{\mathbf{2,0}}}{0,0010}$ &\\
& & & &  $\wts{{\mathbf{1,-1}}}{1,0100}$,  $\wts{{\mathbf{-1,1}}}{1,0010}$ &\\
$D_5(a_1)$ & $A_1$ & $1$ & $D_5$ & $\wtt{0}{2}{100000}$, $\wtt{0}{-2}{10000}$ &
$S_2$ \\
& & & & $\wtt{1}{-1}{000010}$, $\wtt{0}{1}{00001}$ & \\
$E_6(a_3)$ & $A_1$ & 0 & $F_4$ & $\wts{2}{0001}$ & $S_2$ \\
$E_7(a_5)$ & 0 & 0 & $\g$ & & $S_3$\\
$E_6(a_1)$ & 0 & 1 & $E_6$ & $\wts{\mathbf{-2}}{000001}$,
$\wts{\mathbf{2}}{100000}$ & $S_2$\\
$E_7(a_3)$ & 0 & 0 & $\g$ & & $S_2$ \\
\hline  
\end{longtable}

\begin{longtable}{|l|l|l|l|l|l|}
\caption{Nilpotent orbits in the Lie algebra of type $E_8$ whose stabilizer
  has a nontrivial component group}\label{tab:E8} 
\endfirsthead
\hline
\endhead
\hline
\endfoot
\endlastfoot

\hline
label & $\c_1'$ & $\dim \ttt$ & $\c_2'$ & weights of $V$ & $A$\\
\hline
$A_2$ & $E_6$ & 0 & $A_2$ & $\wts{100000}{10}$, $\wts{000001}{01}$ & $S_2$\\
$A_2+A_1$ & $A_5$ & 0 & $A_1+A_2$ & $\wts{10000}{1,10}$, $\wts{00001}{1,01}$ &
$S_2$\\
& & & &  $\wts{01000}{0,01}$, $\wts{00010}{0,10}$ & \\
& & & &  $\wts{00100}{1,00}$ & \\
$2A_2$ & $2G_2$ & 0 & $A_1$ & $\wts{10,10}{2}$, $\wts{10,00}{4}$,
$\wts{00,10}{4}$ & $S_2$\\
$D_4(a_1)$ & $D_4$ & 0 & $D_4$ & $\wts{1000}{1000}$, $\wts{0010}{0010}$,
$\wts{0001}{0001}$ & $S_3$\\
$D_4(a_1)+A_1$ & $3A_1$ & 0 & $A_1+D_4$ & $\wts{1,1,0}{0,1000}$,
$\wts{1,0,1}{0,0010}$, $\wts{0,1,1}{0,0001}$ & $S_3$\\
& & & & $\wts{1,0,0}{1,0001}$,$\wts{0,1,0}{1,0010}$, $\wts{0,0,1}{1,1000}$ &\\
& & & &  $\wts{1,1,1}{1,0000}$ & \\
$A_3+A_2$ & $B_2$ & 1 & $A_2+B_2$ & $\wtt{00}{-4}{01,00}$, $\wtt{00}{4}{10,00}$ &
$S_2$\\
& & & & $\wtt{00}{2}{01,10}$, $\wtt{00}{-2}{10,10}$ & \\
& & & & $\wtt{10}{2}{01,00}$, $\wtt{10}{-2}{10,00}$ & \\
& & & & $\wtt{01}{3}{00,01}$, $\wtt{01}{-3}{00,01}$ & \\
& & & & $\wtt{01}{-1}{01,01}$, $\wtt{01}{1}{10,01}$ & \\
& & & & $\wtt{10}{0}{00,10}$ & \\
$A_4$ & $A_4$ & 0 & $A_4$ & $\wts{1000}{0100}$, $\wts{0001}{0010}$ & $S_2$\\
& & & & $\wts{0010}{1000}$, $\wts{0100}{0001}$ & \\
$D_4(a_1)+A_2$ & $A_2$ & 0 & $A_1$ & $\wts{30}{4}$, $\wts{03}{4}$, $\wts{22}{2}$,
$\wts{11}{6}$ & $S_2$\\
$A_4+A_1$ & $A_2$ & 1 & $A_1+A_4$ & $\wtt{00}{-6}{0,0001}$,
$\wtt{00}{6}{0,1000}$ & $S_2$\\
& & & & $\wtt{01}{-5}{1,0000}$, $\wtt{10}{5}{1,0000}$ &\\
& & & & $\wtt{10}{-4}{0,1000}$, $\wtt{01}{4}{0,0001}$ &\\
& & & & $\wtt{00}{-3}{1,0100}$, $\wtt{00}{3}{1,0010}$ &\\
& & & & $\wtt{10}{-1}{1,0001}$, $\wtt{01}{1}{1,1000}$ &\\
& & & & $\wtt{01}{-2}{0,0010}$, $\wtt{10}{2}{0,0100}$ &\\
$D_5(a_1)$ & $A_3$ & 0 & $D_5$ & $\wts{100}{00010}$, $\wts{001}{00001}$ &
$S_2$\\
&&&& $\wts{010}{10000}$ & \\
$A_4+2A_1$ & $A_1$ & 1  & $A_1+A_4$ & $\wtt{1}{-5}{1,0000}$, $\wtt{1}{5}{1,0000}$
& $S_2$\\
&&&&  $\wtt{0}{-4}{0,0100}$, $\wtt{0}{4}{0,0010}$ & \\
&&&&  $\wtt{2}{-2}{0,1000}$, $\wtt{2}{2}{0,0001}$ & \\
&&&&  $\wtt{0}{-2}{2,1000}$, $\wtt{0}{2}{2,0001}$ & \\
&&&&  $\wtt{1}{-3}{1,0001}$, $\wtt{1}{3}{1,1000}$ & \\
&&&&  $\wtt{1}{-1}{1,0010}$, $\wtt{1}{1}{1,0100}$ & \\
&&&&  $\wtt{2}{0}{2,0000}$ & \\
$D_4+A_2$ & $A_2$ & 0 & $A_2+G_2$ & $\wts{20}{10,00}$, $\wts{02}{01,00}$ & 
$S_2$\\
&&&& $\wts{10}{01,10}$, $\wts{01}{10,10}$ & \\
&&&& $\wts{11}{00,10}$ & \\
$E_6(a_3)$ & $G_2$ & 0 & $F_4$ & $\wts{10}{0001}$ & $S_2$\\
$D_6(a_2)$ & $2A_1$ & 0 & $D_6$ & $\wts{10}{000010}$, $\wts{01}{000001}$ &
$S_2$ \\
&&&& $\wts{11}{000001}$ & \\
$E_6(a_3)+A_1$ & $A_1$ & 0 & $A_1+F_4$ & $\wts{3}{1,0000}$, $\wts{2}{0,0001}$,
$\wts{1}{1,0001}$ & $S_2$\\
$E_7(a_3)$ & $A_1$ & 0 & $E_7$ & $\wts{1}{0000001}$ & $S_3$\\
$E_8(a_7)$ & 0 & 0 & $\g$ & & $S_5$ \\
$D_6(a_1)$ & $2A_1$ & 0 & $D_6$ & $\wts{1,1}{100000}$, $\wts{1,0}{000010}$,
$\wts{0,1}{000001}$ & $S_2$\\
$E_7(a_4)$ & $A_1$ & 0 & $E_7$ & $\wts{1}{0000001}$ & $S_2$ \\
$E_6(a_1)$ & $A_2$ & 0 & $E_6$ & $\wts{10}{000001}$, $\wts{01}{100000}$ &
$S_2$\\
$D_5+A_2$ & 0 & 1 & $A_2+D_5$ & $\wts{\mathbf{-4}}{01,00000}$,
$\wts{\mathbf{4}}{10,00000}$ & $S_2$ \\
&&&&  $\wts{\mathbf{-3}}{00,00001}$, $\wts{\mathbf{3}}{00,00010}$ &\\
&&&&  $\wts{\mathbf{-2}}{10,10000}$, $\wts{\mathbf{2}}{01,10000}$ &\\
&&&&  $\wts{\mathbf{-1}}{01,00010}$, $\wts{\mathbf{1}}{10,00001}$ &\\
$D_7(a_2)$ & 0 & 1  & $D_7$ & $\wts{\mathbf{-2}}{1000000}$, $\wts{\mathbf{2}}{1000000}$ & $S_2$\\
&&&& $\wts{\mathbf{-1}}{0000010}$, $\wts{\mathbf{1}}{0000001}$ &\\
$E_6(a_1)+A_1$ & 0 &  1 & $A_1+E_6$ & $\wts{\mathbf{-3}}{1,000000}$, $\wts{\mathbf{3}}{1,000000}$ & $S_2$\\
&&&&  $\wts{\mathbf{-2}}{0,000001}$, $\wts{\mathbf{2}}{0,100000}$ & \\
&&&&  $\wts{\mathbf{-1}}{1,100000}$, $\wts{\mathbf{1}}{1,000001}$ & \\
$E_7(a_3)$ & $A_1$ & 0 & $E_7$ & $\wts{1}{0000001}$ & $S_2$\\
$E_8(b_6)$ & 0 & 0 & $\g$ & & $S_3$\\
$D_7(a_1)$ & 0 & 1 & $D_7$ & $\wts{\mathbf{-2}}{100000}$, $\wts{\mathbf{2}}{1000000}$ &  $S_2$\\
&&&&  $\wts{\mathbf{-i}}{000001}$, $\wts{\mathbf{i}}{01000010}$ & \\
$E_8(a_6)$ & 0 & 0 & $\g$ & & $S_3$\\
$E_8(b_5)$ & 0 & 0 & $\g$ & & $S_3$\\
$E_8(a_5)$ & 0 & 0 & $\g$ & & $S_2$\\
$E_8(b_4)$ & 0 & 0 & $\g$ & & $S_2$\\
$E_8(a_4)$ & 0 & 0 & $\g$ & & $S_2$\\
$E_8(a_3)$ & 0 & 0 & $\g$ & & $S_2$\\
\hline  
\end{longtable}

%\bibliographystyle{plain}
%\bibliography{../refs}

\end{document}